\documentclass{article}

\usepackage{amsmath,amssymb,amsthm, enumerate, ifthen}

\usepackage{tikz}
\usetikzlibrary{decorations.pathreplacing}
\usetikzlibrary{calc}
\usepackage{url}

\usepackage{thmtools}

\usepackage{framed}
\usepackage{comment}

\newcommand\N{\mathbb N}
\newcommand\Z{\mathbb Z}
\newcommand{\Aut}{\mathrm{Aut}}

\newcommand{\Sy}[1]{
 \tikz[scale=0.3]{
  \foreach \n [count=\i] in {#1} {
   \ifthenelse{\n=0 \OR \n=1 \OR \n=2 \OR \n=3}{
   }{
    \draw[color=black!15!white,fill=black!5!white] (\i,0) rectangle ({\i + 1},1);
   }
  }
  \foreach \n [count=\i] in {#1} {
   \ifthenelse{\n=1}{
    \draw (\i,0.5) -- ({\i + 1},0.5);
   }{}
   \ifthenelse{\n=2}{
    \draw ({\i + 0.5},0) -- ({\i + 0.5},1);
   }{}
   \ifthenelse{\n=3}{
    \draw (\i,0.5) -- ({\i + 1},0.5);
    \draw ({\i + 0.5},0) -- ({\i + 0.5},1);
   }{}
  \ifthenelse{\n=5}{
    \draw (\i,0.5) -- ({\i + 1},0.5);
   }{}
   \ifthenelse{\n=6}{
    \draw ({\i + 0.5},0) -- ({\i + 0.5},1);
   }{}
   \ifthenelse{\n=7}{
    \draw (\i,0.5) -- ({\i + 1},0.5);
    \draw ({\i + 0.5},0) -- ({\i + 0.5},1);
   }{}
   \ifthenelse{\n=0 \OR \n=1 \OR \n=2 \OR \n=3}{
    \draw[color=black!15!white] (\i,0) rectangle ({\i + 1},1);
   }{}
  }
 }
}

\numberwithin{equation}{section}
\theoremstyle{plain}
\newtheorem{lemma}[equation]{Lemma}
\newtheorem{theorem}[equation]{Theorem}
\newtheorem{proposition}[equation]{Proposition}

\newtheorem{question}[equation]{Question}

\newtheorem{remark}[equation]{Remark}

\theoremstyle{definition}
\newtheorem{definition}[equation]{Definition}

\newtheorem{example}[equation]{Example}

\title{Transitive action on finite points of a full shift and a finitary Ryan's theorem}

\author{Ville Salo}

\begin{document}
\maketitle

\begin{abstract}
We show that on the four-symbol full shift, there is a finitely generated subgroup of the automorphism group whose action is (set-theoretically) transitive of all orders on the points of finite support, up to the necessary caveats due to shift-commutation. As a corollary, we obtain that there is a finite set of automorphisms whose centralizer is $\Z$ (the shift group), giving a finitary version of Ryan's theorem (on the four-symbol full shift), suggesting an automorphism group invariant for mixing SFTs. We show that any such set of automorphisms must generate an infinite group, and also show that there is also a group with this transitivity property which is a subgroup of the commutator subgroup and whose elements can be written as compositions of involutions. We ask many related questions and prove some easy transitivity results for the group of reversible Turing machines, topological full groups and Thompson's~$V$.
\end{abstract}

\section{Introduction}

When $\Sigma$ is a finite set, $\Sigma^\Z$ is homeomorphic to the Cantor set and under the self-homeomorphism $\sigma(x)_i = x_{i+1}$ becomes a dynamical $\Z$-system called the \emph{full shift}. Our main result is about the full shift with $\Sigma = \{0,1,2,3\}$, but the motivation comes from the more general setting of sofic shifts and more specifically \emph{mixing SFTs}, which are the topologically mixing subsystems of full shifts defined by a finite set of forbidden patterns \cite{LiMa95}. 
The \emph{automorphism group} of a subshift $X$ is the set of homeomorphisms $f : X \to X$ that commute with $\sigma$. It is indeed a group under function composition, and it acts on $X$ in the obvious way.

Forgetting the action of the automorphism group of a mixing SFT,\footnote{Mixing SFTs do not all have isomorphic automorphism groups, but the singular `group' refers to a typical example; these groups are qualitatively similar in many key aspects.} it becomes an interesting abstract group \cite{He69,BoLiRu88,KiRo90}. Many group-theoretic questions about it are undecidable \cite{KiRo90,KaOl08,Sa16a}. Ryan's theorem \cite{Ry72} states that the center of this group is as small as it possibly could be, namely the group of shifts $\langle \sigma \rangle$. More generally, in \cite{FrScTa17} it is shown that normal amenable subgroups consist of shift maps.

The finitely generated subgroups of an infinitely generated group are in many respects the essence of the group.\footnote{In group theory, properties are called local if they are determined by finitely-generated subgroups. Many properties of interest are local, for example amenability, residual finiteness and soficness.} Much is known about the set of finitely generated subgroups of $\Aut(X)$ for mixing SFTs $X$ as abstract groups: It is closed under subgroups, (countable) free and direct products \cite{Sa16a} and extensions by finite groups \cite{KiRo90} and contains the right-angled Artin groups (also known as `graph groups') \cite{KiRo90,Ch07}.\footnote{Note that the closure properties alone imply that all f.g. abelian groups and all f.g. nonabelian free groups are subgroups: since the trivial group is a subgroup, using finite group extensions, we get all finite groups. With free products, we then get $\Z$, from which we get all f.g. abelian and free groups using direct and free products, respectively.} On the side of limitations, we know that all groups it contains are residually finite, and finitely generated subgroups have a decidable word problem \cite{BoLiRu88}.

The group has been studied also through its action. Some of the most important developments have been on the dimension representation and the action on finite subsystems of $\Sigma^\Z$. See \cite{BoLiRu88,BoFi91,KiRoWa00,KiRoWa00a,Bo08}. It is also known that the action is `as topologically transitive as possible', in the sense that every aperiodic point has dense orbit \cite{BoLiRu88}. See also \cite{LiMa95} for a discussion of this group.

Actions of individual automorphisms (as $\Z$-actions) have been studied quite a bit, and in particular their expansivity is a very interesting topic \cite{Na95,Bo04,Na08}. Many results about possible dynamics and undecidability results are known for these actions, and are often proved under the name reversible cellular automata \cite{Ka05,KaOl08,Lu10a}.

\section{The results}

In this article, we begin the study of actions of finitely generated (noncommutative) subgroups of $\Aut(X)$. We construct a particular finitely generated subgroup of $\Aut(X)$ for a particular mixing SFT $X$ with an interesting action, 
namely we find an action of a finitely generated group $G$ on $\{0,1,2,3\}^\Z$ by automorphisms, which is as transitive as possible on the \emph{nonzero finite points}, that is, points having finite and nonempty support. More precisely, for any $k$ and every pair of $k$-tuples $(x_1,\ldots,x_k)$ and $(y_1,\ldots,y_k)$ of nonzero points of finite support, there is an element $f \in G$ such that $(y_1,\ldots,y_k) = (f(x_1),\ldots,f(x_k))$, assuming that $x_i$ and $x_j$ (resp. $y_i$ and $y_j$) come from different (shift) orbits when $i \neq j$. In the terminology of the following section, this amounts to the following:

\begin{theorem}
\label{thm:Main}
For $\Sigma = \{0,1,2,3\}$ there is a finitely generated subgroup of $\Aut(\Sigma^\Z)$ that acts $\infty$-orbit-transitively on the set of nonzero finite points. 
\end{theorem}

We prove this in Theorem~\ref{thm:MainProof}. The finite points are a natural set to study, since they are preserved by the automorphism group\footnote{More precisely, the finite-index subgroup of the automorphism group that stabilizes the point $0^\Z$ -- or alternatively the automorphism group of the corresponding $0$-pointed subshift.} and form a countable set.

Our proof is explicit in the sense that we give a finite list of automorphisms, and show how to turn a tuple of nonzero finite points into another one by a finite composition of them. 
Proving this is equivalent to proving that we can perform any permutation of any finite set of nonzero finite points from different orbits. Thus it is also equivalent to the fact that given any set of nonzero finite points from different orbits, we can transpose two of them without changing the others. (These observations are not used in the proof of the main result, but see Lemma~\ref{lem:InftyOrbitCharacterization} for a proof.)

An interesting corollary of this theorem is the following stronger version of Ryan's theorem \cite{Ry72} (on the four-symbol full shift):

\begin{theorem}
\label{thm:Ryan}
For $\Sigma = \{0,1,2,3\}$ there is a finite set $F \subset \Aut(\Sigma^\Z)$ such that for $g \in \Aut(\Sigma^\Z)$, we have $g \in \langle \sigma \rangle$ if and only if $\forall f \in F: f \circ g = g \circ f$.
\end{theorem}

Without further ado, let us prove this using Theorem~\ref{thm:Main}.

\begin{proof}
Let $F$ be the set of generators of the group in Theorem~\ref{thm:Main} together with all the cellwise symbol-permutations. If $g$ is not a shift, then if $g(0^\Z) \neq 0^\Z$, $g$ does not commute with some symbol-permutation. Otherwise, there is a nonzero finite point $x \in \{0,1,2,3\}^\Z$ such that $g(x)$ is not in the shift orbit of $x$ by Lemma~\ref{lem:gxNotAShift}. Then there exists $f \in \langle F \rangle$ mapping $(x, g(x))$ to $(\sigma(x), g(x))$ by $2$-orbit-transitivity. Then $g(f(x)) = \sigma(g(x)) \neq g(x) = f(g(x))$. Since $g$ does not commute with $f$, it cannot commute with all maps in $F$.
\end{proof}

We also prove the following stronger version of Theorem~\ref{thm:Main}. This is shown in Theorem~\ref{thm:InvoCommu}.

\begin{theorem}
For $\Sigma = \{0,1,2,3\}$ there is a finite set $F$ such that $\langle F \rangle$ acts $\infty$-orbit-transitively on the set of nonzero finite points $x \in \Sigma^\Z$, and each element of $F$ is a commutator of two automorphisms, and a composition of involutive (= order two) automorphisms.
\end{theorem}

We also prove a version of Ryan's theorem for such a set, see Theorem~\ref{thm:RyanVersion}.


While $F$ is finite in the statement of the theorem, the group $\langle F \rangle$ is infinite. This is necessary by the following result, whose proof we give in Theorem~\ref{thm:CentralizerOfFiniteGroup}.

\begin{restatable}{theorem}{CentralizerOfFiniteGroup}
\label{thm:CentralizerOfFiniteGroup}
Let $G$ be any finite group of automorphisms on a mixing SFT $X$, and $\Gamma$ any finite alphabet. Then there is an embedding $\phi : \Aut(\Gamma^\Z) \to \Aut(X)$ such that $\phi(f)$ commutes with $g$ for all $f \in \Aut(\Gamma^\Z)$ and $g \in G$.
\end{restatable}

In the context of finite permutation groups, it is known that the only faithful $k$-transitive actions for $k \geq 6$ are those of symmetric groups and alternating groups \cite[Section~7.3]{DiMo96}. The proof is a case analysis based on the classification of finite simple groups. For infinite groups, high-order transitivity is presumably quite common. We give two examples of $\infty$-transitivity in other contexts that have a symbolic dynamical interpretation:

\begin{example}
Thompson's~$V$ is defined by its action on the interval. This action is well-defined on the countable set of dyadic rationals, and on this set the action is $\infty$-transitive.
\end{example}

\begin{example}
The topological full group of a minimal subshift is defined by its action on the subshift. This action is well-defined on the (countable) shift-orbit of every point. There is a finitely generated subgroup of this group, namely its commutator subgroup, that is $\infty$-transitive on every shift-orbit.
\end{example}

See Section~\ref{sec:Notes} for the (easy) proofs, and \cite{CaFlPa96} for a reference for Thompson's~$V$. As an auxiliary result, in Lemma~\ref{lem:TMTransitive} we also prove that a certain natural group of homeomorphisms on the full shift $A^\Z$ called $G_0$, which corresponds to the reversible one-state oblivious Turing machines of \cite{BaKaSa16}, acts $\infty$-transitively on finite points.

\section{Questions}

A result similar to Theorem~\ref{thm:Ryan} is known for the endomorphism monoid of a full shift \cite{SaTo14b}, elaborated on in \cite[page~151]{Sa14}, namely that there is a finite set of endomorphisms of any full shift $\Sigma^\Z$ with $|\Sigma| \geq 3$ whose common centralizer consists of the shift maps.\footnote{In \cite{Sa14}, it is shown that a set of size $|\Sigma|$ suffices to get a trivial centralizer, but a trivial modification of the proof gives an upper bound of $3$.} Note that on any mixing SFT, the set $F \subset \Aut(\Sigma^\Z)$ in Theorem~\ref{thm:Ryan} must be of size at least $2$ and that the minimal cardinality of such $F$ is an isomorphism invariant for $\Aut(X)$, implying that computing it could theoretically separate $\Aut(X)$ and $\Aut(Y)$ for some non-flip-conjugate mixing SFTs $X, Y$ without roots (which is an open problem \cite{Bo08}).

\begin{definition}
\label{def:k}
For a subshift $X$, let $k(X) \in \N \cup \{\infty, \bot\}$ be the minimal cardinality of a set $F$ such that for $g \in \Aut(\Sigma^\Z)$, we have $g \in \langle \sigma \rangle$ if and only if $\forall f \in F: f \circ g = g \circ f$, if such a set exists, and $k(X) = \bot$ otherwise.
\end{definition}

Our result implies that $k(\{0,1,2,3\}^\Z) \in \N$. See Section~\ref{sec:Notes} for a discussion of this invariant and related questions.

In particular with this application in mind, it is interesting to ask how small we can make the subgroup of $\Aut(\Sigma^\Z)$ in the proof of Theorem~\ref{thm:Main}. The one in our proof has a reasonable\footnote{At most eight, see Section~\ref{sec:Notes}.} number of generators, but presumably more than are needed.

\begin{question}
In Theorem~\ref{thm:Main}, how many generators do we need? Can we pick the finitely generated subgroup to have only two generators?
\end{question}

One automorphism is not enough, and more generally abelian actions cannot be $\infty$-orbit-transitive (see Proposition~\ref{prop:AbelianNotOrbitTransitive}). In fact, no individual automorphism is even transitive on nonzero finite points,\footnote{If $f^p(x) = \sigma(x)$ for some $p$, then $\{f^n(x) \;|\; n \in \Z\}$ intersects finitely many $\sigma$-orbits.} but we cannot show that one cannot be transitive on the set of orbits of nonzero finite points. 

\begin{question}
Does there exist an infinite pointed mixing SFT $X$ and $f \in \Aut(X)$ such that $\langle \sigma, f \rangle$ acts (set-theoretically) transitively on the set of nonzero finite points?
\end{question}

A transitive action of $\langle \sigma, f \rangle$ on nonzero finite points would look rather interesting. It would necessarily have to be free, and would give every nonzero finite point $y$ a unique coordinate $(m,n) \in \Z^2$ by $y = \sigma^m \circ f^n(...000w000...)$, where $...000w000... \in X$ is chosen arbitrarily.

In this context, we should mention the closely related result of Kari that, up to shifting, automorphisms can be {\textit{topologically}} transitive on the set of nonzero finite points: 

\begin{theorem}[\cite{Ka12a}]
\label{thm:Kari}
Let $\Sigma = \{0,1,2,3,4,5\}$. Then there exists $f \in \Aut(\Sigma^\Z)$ such that $\langle \sigma, f \rangle$ acts topologically transitively on the set of nonzero finite points.
\end{theorem}

Just like our proof requires the alphabet size to be composite, $6$ here comes from the fact that it is the smallest product of distinct primes, and the result seems to be open for the binary full shift and non-full mixing SFTs.

Theorem~\ref{thm:Kari} says in particular that a $\Z^2$-action by automorphisms can be topologically transitive on the set of nonzero finite points. It is an open question whether the $\Z$-action given by an individual automorphism can do the same. We ask the question more generally for pointed mixing SFTs.

\begin{question}[\cite{Ka12a}]
\label{q:Kari}
Does there exist an infinite pointed mixing SFT $X$ and $f \in \Aut(X)$ such that $\langle f \rangle$ acts topologically transitively on the set of nonzero finite points?
\end{question}

Often results about finite points are really results about pairwise asymptotic points in disguise (or the other way around). We note that Theorem~\ref{thm:Main} is not equivalent to being able to perform an arbitrary permutation on a set of points all (left and right) asymptotic to a given point; in fact even the full automorphism group cannot have this property since if $(x_1, \ldots, x_k)$ are all asymptotic to a point whose forward orbit is dense, then any nontrivial permutation of the points will necessarily change the tails. Of course one can ask whether we could take any set of mutually asymptotic points and have transitivity up to this `obvious' restriction:

\begin{question}
Let $X$ be a mixing SFT. Is there a finitely generated subgroup $G \leq \Aut(X)$ such that if $(x_1, \ldots, x_k)$ are pairwise asymptotic and each $x_i$ contains a word $w_i$ that occurs only once in $x_i$ and does not occur in $x_j$ for $j \neq i$, every permutation of the $x_i$ can be performed by the $G$-action?
\end{question}

Our construction fails to prove this result for at least two reasons. First, we use the fact that the four-symbol full shift is decomposable into a Cartesian product in a nontrivial way, as we use this in the proof for introducing Turing machine heads. Some\footnote{I would like to say `most', but there are many concrete ways to sample SFTs at random (see in particular \cite{Mc12}), and I do not know how uniformly the number-theoretic properties of entropies (such as primality of the corresponding Perron numbers) distribute.} mixing SFTs, in particular the two-symbol full shift, do not have this property \cite{Li84a,Bo84}. Another difficulty is that in the proof that the action of the group $G_0$ is transitive on finite points (see Definition~\ref{def:OB} and Lemma~\ref{lem:TMTransitive}), we use the fact that the full shift is closed under permutations of cells -- a property that characterizes full shifts among transitive subshifts. It seems likely that these difficulties can be addressed with careful application of marker methods.

A more natural way to prove Theorem~\ref{thm:Main} would be to find a natural subgroup of $\Aut(X)$ (in terms of the action or properties as a subgroup), and separately show finitely-generatedness and the transitivity property as theorems. To the best of our knowledge, the commutator subgroup could be a candidate for this. It is not difficult to show that the commutator subgroup satisfies the statement of Theorem~\ref{thm:Main}, and we show this in Theorem~\ref{thm:CommutatorAction}. However, we do not know if it is finitely generated.

\begin{question}
\label{q:Commutator}
For which mixing SFTs $X$ is the commutator subgroup of $\Aut(X)$ finitely generated? 
\end{question}

\section{Definitions and conventions}

Our convention is $0 \in \N$. We do not need real intervals, so by $[a,b]$ we mean the discrete interval $[a,b] \cap \Z$. Positions in words are indexed starting from $0$.

A \emph{word} is a (possibly empty) list of symbols over a finite set, or \emph{alphabet} $\Sigma$. Formally, we think of words as elements of the free monoid $\Sigma^*$ generated by $\Sigma$, with concatenation as the monoid operation. Elements of $\Sigma^\Z$ (functions from $\Z$ to $\Sigma$) are called \emph{configurations} or \emph{points} and for $x \in \Sigma^\Z$ and $i \in \Z$ we write $x_i$ instead of $x(i)$. For $a,b \in \Z$ we write $x_{[a,b]}$ for the subword $x_a x_{a+1} \cdots x_b$ of $x$. A point $x \in \Sigma^\Z$ is sometimes denoted as $y.z$ where $y \in \Sigma^{(-\infty,-1]}$ and $z \in \Sigma^\N$, and $x_i = y_i$ and $x_i = z_i$ on the respective domains. If $a \in \Sigma$ and $w \in \Sigma^*$, we write $|w|_a = |\{i \;|\; w_i = a\}|$.

All groups in this paper are countable discrete groups, unless otherwise mentioned. If $G$ is a group, a \emph{$G$-set} is a set $X$ equipped with a (left) action of $G$ by bijections. A $G$-set where all $g \in G$ act by continuous maps is a \emph{(dynamical) $G$-system}. 

\begin{definition}
Let $X$ be any topologically closed subset of $\Sigma^\Z$ in the product topology, where $\Sigma$ is finite, such that $\sigma(X) = X$ where $\sigma$ is the \emph{shift} defined by $\sigma(x)_i = x_{i+1}$ for all $i \in \Z$. Then $X$ is called a \emph{subshift}, and the homeomorphism $\sigma$ makes it a $\Z$-system. If a homeomorphism $f : X \to X$ commutes with $\sigma$, we call it an \emph{automorphism} of $X$. We write $\Aut(X)$ for the automorphism group of a subshift $X$. A \emph{pointed subshift} is a pair $(X, x)$ where $x = a^\Z \in X$ for some $a \in \Sigma$. Then $a$ is called the \emph{zero (symbol)} and $x$ the \emph{zero (point)}. Automorphisms of pointed subshifts are the automorphisms of the underlying subshift that additionally preserve the zero point.
\end{definition}

Our zero $a$ will be $0$ in text and $\Sy{0}$ in figures.

\begin{definition}
Let $0 \in \Sigma$ and let $X \subset \Sigma^\Z$ be a pointed subshift with zero $0^\Z$. A point $x \in X$ is called \emph{finite}\footnote{Such points are examples of the \emph{homoclinic} points in continuous dynamics. The term `homoclinic' is used as a synonym of our `finite' in for example \cite{LiSc99} in the context of subshifts with algebraic structure, but it can also mean a point of the form $...uuu w uuu...$ where $u$ is a word of not necessarily unit length. The term `finite point' is commonly used in the theory of cellular automata.} if the support $\{i \;|\; x_i \neq 0\}$ is a finite set. Write $X_0 = \{ x \in X \;|\; x \mbox{ is finite and } x \neq 0^\Z \}$ for the set of nonzero finite points and $X_0' = X_0 \cup \{0^\Z\}$ for all finite points.
\end{definition}

We now define our notions of transitivity for Theorem~\ref{thm:Main}. 

If $H$ is a group, $X$ is an $H$-set and $Z \subset X$, then write
\[ Z^{(k)} = \{ (z_1, \ldots, z_k) \in Z^k \;|\; z_i \in H z_j \implies i = j \}. \]
Thus $Z^{(k)}$ is the $k$th Cartesian power of $Z$ with the additional restriction that coordinates should be from different $H$-orbits (and $H$ and its action are clear from context).

\begin{definition}
Let $X$ be an $H$-set, and let $G$ act on $X$ by $H$-commuting maps and $Z \subset X$. We say the action of $G$ is \emph{transitive around $Z$} if $\forall y,z \in Z: y \in G z$. We say $G$ \emph{$k$-orbit-transitive around} $Z$ if the diagonal action of $G$ on $X^{(k)}$ is transitive around $Z^{(k)}$. If the action of $G$ is $k$-orbit-transitive (around a set) for all $k$, then we say it is \emph{$\infty$-orbit-transitive} (around that set).
\end{definition}

Note that when we say \emph{$G$ is transitive around $Z$} we do not necessarily imply that $G$ has a well-defined action on $Z$, but on some $X \supset Z$. This is a technical notion that is useful for stating the intermediate steps of the proof -- our main result will be about the usual kind of (set-theoretic) transitivity: \emph{$G$ acts transitively on $Z$} if $G Z \subset Z$ and $G$ acts transitively around $Z$.

In our application, $X$ is the subshift, the group $H$ is the $\Z$-action of $\sigma$, and $G$ the finitely generated subgroup of the automorphism group. For $x \in \Sigma^\Z$, write the $\sigma$-orbit of $x$ as $\mathcal{O}(x) = \{\sigma^n(x) \;|\; n \in \Z\}$ instead of $\Z x$. For a set of points $Y$, write $\mathcal{O}(Y) = \bigcup_{y \in Y} \mathcal{O}(y)$.


In the proofs, we deal with many subgroups of $\Aut(X)$, and for $F_1,F_2,\ldots,F_\ell$ individual automorphisms or sets of them, we write $\langle F_1,\ldots,F_k \rangle$ for the smallest subgroup of $\Aut(X)$ containing all automorphisms in the list, or in the sets $F_i$. {\textit{We do not need presentations, so we use the following non-standard `group-builder notation'}} when $F$ is a set of elements of a group and $P$ is a property of elements of that group:
\[ \langle a \;|\; a \in F, P(a) \rangle = \langle \{a \;|\; a \in F, P(a)\} \rangle \]
instead of the usual meaning of $\langle A \;|\; R \rangle$ where $R$ is a set of relations. This should cause no confusion.











\section{The generators and transitivity results}

In this section, we define the finite set of automorphisms generating the group $\mathcal{G}$ acting $\infty$-orbit-transitively in Theorem~\ref{thm:Main}. 
We also prove some auxiliary transitivity results. Namely, we prove that the automorphism group of a full shift is transitive around certain marked configurations, and we also define the `reset system', a transitive monoid action used in the proof of the main theorem. 

Most of the hard work in the main proof is performed by the action of the group $\mathcal{G}_{s}$, which is constructed in several steps in the subsections below but we will also need the following simpler transformations:

\begin{definition}
Let $A, B$ be finite alphabets. The \emph{$(A, B)$-particle rule} is the automorphism $P$ of $(A \times B)^\Z$ defined by $P(x, y) = (\sigma(x), y)$. A \emph{symbol-permutation} is an automorphism $g_\pi : \Sigma^\Z \to \Sigma^\Z$ defined by a permutation $\pi : \Sigma \to \Sigma$ by $g_{\pi}(x)_i = \pi(x_i)$.
\end{definition}

\subsection{`Gates and Turing machines'}

The idea of transitivity around marked configurations, which we will prove in the next section, comes from \cite{BaKaSa16}, where it was (implicitly) shown that there is a finitely generated subgroup of the group of reversible (generalized) Turing machines which acts $\infty$-transitively around configurations where the head of the machine is at the origin. This is further based on the study of `reversible clones' \cite{BoKaSa16,Bo15} or equivalently `monoidal groupoids' \cite{La03,Se16} of bijections on $A^*$, where $A$ is a finite alphabet.

We assume no familiarity with these notions, and to keep terminology to a minimum, in this section we only deduce the necessary facts, based on the following lemma, which is a direct corollary of \cite[Theorem~20]{BoKaSa16}:

\begin{lemma}
\label{lem:Gates}
Let $A$ be any finite alphabet. Then there exists $n$ such that for all even permutations $\pi : A^m \to A^m$ with $m \geq n$, we have $\pi = \alpha_1 \circ \alpha_2 \circ \cdots \circ \alpha_\ell$ where for all $i \in \N$, $\alpha_i : A^m \to A^m$ and there exists $k_i \in [0, m-n]$ and an even permutation $\beta_i : A^n \to A^n$ such that
\[ \alpha_i(uvw) = u \cdot \beta_i(v) \cdot w \]
for all $u \in A^{k_i}, v \in A^n, w \in A^{m-k_i-n}$.
\end{lemma}

In \cite[Theorem~20]{BoKaSa16}, this is proved with $n = 4$. In the case when $|A|$ is odd, $n = 2$ suffices \cite{Bo15,Se16}.

\begin{definition}
\label{def:OB}
Let $N \subset \Z$ be finite. If $\pi : A^N \to A^N$ is a permutation, then the homeomorphism $P_{\pi} : A^\Z \to A^\Z$ defined by
\[
P_{\pi}(x)_i = \left\{ \begin{array}{cc}
x_i & \mbox{if } i \notin N \\
\pi(x_N)_i & \mbox{otherwise.}
\end{array}\right.
\]
is called a \emph{local permutation} 
 with \emph{neighborhood $N$}. Define
\[ G_0'' = \langle P_\pi \;|\; \exists \mbox{ finite } N \subset \Z: \pi : A^N \to A^N \mbox{ is an even permutation} \rangle, \]
\[ G_0' = \langle P_\pi \;|\; \exists \mbox{ finite } N \subset \Z: \pi : A^N \to A^N \mbox{ is a permutation} \rangle, \]
\[ G_0 = \langle \sigma, g \;|\; g \in G_0' \rangle. \]
\end{definition}

Thus by $G_0$ we denote group of self-homeomorphisms of $A^\Z$ generated by $\sigma$ and the local permutations (of all finite neighborhoods). 

Note that $G_0'' = G_0'$ when $|A|$ is even, since every permutation becomes even when its neighborhood is increased. On the other hand, when $|A|$ is odd, one can consistently assign a parity to each $g \in G_0'$, and we obtain $[G_0' : G_0''] = 2$. Every element of $G_0$ is clearly of the form $\sigma^\ell \circ g$ where $g \in G_0'$, $\ell \in \Z$. The group $G_0'$ is a direct union of groups of local permutations with fixed neighborhoods.

In the terminology of \cite{BaKaSa16}, the group $G_0$ is denoted $\mathrm{OB}(\Z,n,1)$ and is called the group of oblivious one-state Turing machines in the moving tape model.

\begin{lemma}
\label{lem:TMTransitive}
Let $X = A^\Z$ with $|A| \geq 2$ and let $G_0''$, $G_0'$ and $G_0$ be as above. Then the following hold:
\begin{itemize}
\item $G_0$ is finitely generated, 
\item $G_0'' = [G_0,G_0]$ is a locally finite simple group, 
\item $G_0''$ acts $\infty$-transitively on the set of finite points $X_0'$, and 
\item there is a unique homomorphism $\psi : G_0 \to \Z$ with $\psi(\sigma) = 1$, $\psi(G_0') = \{0\}$, and we have $G_0' = \ker(\psi)$.
\end{itemize}
\end{lemma}

\begin{proof}
First, we prove finitely generatedness of $G_0$. Let $H$ be the subgroup of $G_0$ generated by $\sigma$ and the local permutations with neighborhood $\{0,1,\ldots,n-1\}$, where $n$ satisfies the claim of Lemma~\ref{lem:Gates} and $n \geq 2$.

We will show $H = G_0$. It is enough to show that our generators generate $G_0'$, as $\sigma$ is in our generating set. For this, first let $\pi : A^N \to A^N$ be any even permutation, where $N = \{0,1,\ldots,m-1\}$ for some $m \geq n$. We can apply arbitrary finite support permutations to the coordinates of a point, because the swaps $(n \;\; n+1)$ generate the symmetric groups, and because our set of generators contains the shift map and the `coordinate swap' $P_{\gamma}$, where $\gamma : A^{\{0,1\}} \to A^{\{0,1\}}$ is defined by $\gamma(ab) = ba$ for all $a,b \in A$. By conjugating with a permutation of coordinates, we can apply any permutation of $A^n$ to any ordered subset of coordinates in $N$ with size at most $n$. Since $\pi$ is even, by Lemma~\ref{lem:Gates}, $\pi$ can be decomposed into such permutations, and thus so can $P_\pi$.

Suppose then that $\pi : A^N \to A^N$ is an odd permutation. If $|A|$ is even, $\pi$ can be seen as an even permutation on $A^{N \cup \{m\}}$ by ignoring the coordinate $m$, and is thus in $H$. If $|\Sigma|$ is odd, then let $a \neq b$ be two elements of $A$, and consider the transposition $(a \; b)$ on $\Sigma$. Extending this to a permutation $\pi' : A^N \to A^N$ by permuting only the $0$th coordinate, we obtain an odd permutation that is clearly in $H$. Then $\pi \circ \pi' : A^N \to A^N$ is even, and thus in $H$. This concludes the proof that $H = G_0$.

Next, we show $G_0'' = [G_0, G_0]$. Consider first $P_\pi \in G_0''$ where $\pi : A^N \to A^N$ is an even permutation and $|N| \geq 3$. Then $\pi$ is in the commutator subgroup of the symmetric group on $A^N$, from which it easily follows that $P_\pi \in [G_0, G_0]$. Let then $g = \sigma^{\ell_1} \circ g'$ and $h = \sigma^{\ell_2} \circ h'$ where $g', h' \in G_0'$. Then
\begin{align*}
ghg^{-1}h^{-1} &= (\sigma^{\ell_1} \circ g') \circ (\sigma^{\ell_2} \circ h') \circ (g'^{-1} \circ \sigma^{-\ell_1}) \circ (h'^{-1} \circ \sigma^{-\ell_2}) \\
&= \sigma^{\ell_1 + \ell_2} \circ g'' h' g'^{-1} h''^{-1} \circ \sigma^{-\ell_1 - \ell_2}
\end{align*}
where $g''$ and $h''$ are conjugate by a power of the shift to $g'$ and $h'$, respectively, and $g'' h' g'^{-1} h''^{-1}$ is a local permutation $P_\pi$ for some $\pi : A^{N'} \to A^{N'}$ for some $N' \subset \Z$. If $|A|$ is even, and we pick $|N'|$ large, then this is an even permutation, and we are done. If $|A|$ is odd, then $g''$ and $h''$ have the same parity as permutations on $A^{N'}$ as $g'$ and $h'$, respectively, and again $g'' h' g'^{-1} h''^{-1}$ is even, implying that $ghg^{-1}h^{-1} \in G_0''$.

The local finiteness of $G_0''$ is obvious. For simplicity, let $P_\pi \in G_0''$ where $\pi : A^N \to A^N$ is a nontrivial even permutation. Then for any $N' \supset N$, the domain extension $\pi : A^{N'} \to A^{N'}$ (ignoring the coordinates $N' \setminus N$) is an even permutation of $A^{N'}$. Since the alternating group on $A^{N'}$ is simple if $|A^{N'}| \geq 5$, we can write any even permutation $\pi' : A^{N'} \to A^{N'}$ as a composition of conjugates of $\pi$. It follows that the same is true for any $P_{\pi'} \in G_0''$.

Next, let us show that $G_0''$ acts $\infty$-transitively on $X_0'$. Let $\vec x, \vec y \in (X_0')^{k}$ be such that the components of $\vec x$ are pairwise distinct, and assume the same for $\vec y$. Let $r$ be such that the supports of $\vec x_i$ and $\vec y_i$ are contained in $[-r,r]$ for all $i \in [1,k]$. Let $u_i = (\vec x_i)_{[-r-1,r+1]}$ and $v_i = (\vec y_i)_{[-r-1,r+1]}$. Then there is an even permutation of $A^{2r+3}$ that maps $u_i$ to $v_i$ for all $i$. The corresponding local permutation with $N = [-r-1,r+1]$ maps $x_i$ to $y_i$ for all $i$, and is an element of $G_0''$.

The existence and uniqueness of $\psi : G_0 \to \Z$ is proved as follows: To $g \in G_0$ associate $j \in \Z$ such that $g(x)_i = x_{i+j}$ for all large enough $|i|$ and all $x \in A^\Z$. Clearly the choice of such $j$ is unique, and this becomes a homomorphism from $G_0 \to \Z$ with the required properties. Uniqueness follows because $\sigma$ and $G_0'$ generate $G_0$.

Next we show $G_0' = \ker(\psi)$. We have $G_0' \subset \ker(\psi)$ by definition. To see that $\ker(\psi) \subset G_0'$, observe that if $\psi(g) = 0$ then $g(x)_i = x_i$ for $|i|$ large enough, so $g$ is a local permutation, showing $\ker(\psi) \subset G_0'$.  
\end{proof}

\subsection{Transitivity on marked configurations}

The idea is next to apply the previous lemma to show that the automorphism group of $\Sigma^\Z$ acts $\infty$-transitively around configurations where for some $s \in \Sigma$, there is only one occurrence of $s$, which is at the origin. To do this, we will think of $s$ as the head of a Turing machine (so that it marks the zero-cell of a configuration of $(\Sigma \setminus \{s\})^\Z$ where $G_0$ acts), translate local permutations to local changes near $s$, and turn $\sigma \in G_0$ into movements of the head $s$. The difficult part is moving $s$ with an automorphism, and most of the work in this section is about proving Lemma~\ref{lem:SimulatedShift} which accomplishes this.\footnote{In on older draft \cite{Sa16ArxivOld}, I left it open whether Lemma~\ref{lem:SimulatedShift} is true, and proved the result in a different way. The cellular automaton $f_{\sigma,s}$ used now was suggested by Jarkko Kari. The basic idea is the same, but we transpose different kinds of word pairs. The construction is written in much more detail than in the previous version, because there was a mistake.}

Let $\Sigma = \{0,1,\ldots,n-1\}$ and let $u, v \in \Sigma^k$. Then $u$ and $v$ are \emph{$\ell$-boundary-equivalent} if $u_{[0,\ell-1]} = v_{[0,\ell-1]}$ and $u_{[k-\ell,k-1]} = v_{[k-\ell,k-1]}$. For $x \in \Sigma^\Z$ and $U \subset \Sigma^k$, define $\chi_U(x) \subset \Z$ by $i \in \chi_U(x) \iff x_{[i,i+k-1]} \in U$.

\begin{definition}
\label{def:Safety}
Let $X \subset \Sigma^\Z$ be an SFT with window size at most $h-1$. Let $U \subset \Sigma^k$ and $V \subset \Sigma^h$ for some $h \leq k$. Then $U$ is \emph{$V$-safe (for $X$)} if the following three conditions hold: First, every word $u \in U$ contains a word of $V$ as a subword. Second, whenever $u, v \in U$ are $(h-1)$-boundary-equivalent and $x \in X$ satisfies $x_{[0,k-1]} = u$, we have
\[ \chi_U(x) = \{0\}, \chi_V(x) \subset [0,k-h] \implies \chi_U(x_{(-\infty, -1]}. v x_{[k,\infty)}) = \{0\}. \]
Third, for each $u \in U$ there is at least one point $x \in X$ such that $\chi_U(x) = \{0\}$, $\chi_V(x) \subset [0,k-h]$ and $x_{[0,k-1]} = u$.
\end{definition}

Note that in particular $U$ must be contained in the language of $X$ by the third condition.

If $U \subset \Sigma^k$, a function $\pi : U \to U$ is \emph{$\ell$-safe} if $u$ and $\pi(u)$ are $\ell$-boundary-equivalent for all $u \in U$.

Safe bijections on safe sets of words give rise to automorphisms by local rewriting, provided that we are very selective in where we apply the rewriting rule:

\begin{lemma}
\label{lem:TechnicalSafety}
Let $X \subset \Sigma^\Z$ be an SFT with window size at most $h-1$. Let $U \subset \Sigma^k$ and $V \subset \Sigma^h$ for some $h \leq k$, and suppose $U$ is $V$-safe. Let $\Gamma$ be the group of $(h-1)$-safe permutations of $U$. Then there is a group monomorphism $\gamma : \Gamma \to \Aut(X)$ such that, letting
\[ Y_1 = \{ x \in X \;|\; \chi_U(x) = \{0\}, \chi_V(x) \subset [0,k-h] \}, \]
we have
\[ \forall x \in Y_1: \forall \pi \in \Gamma: \gamma(\pi)(x) = x_{(-\infty, -1]}. \pi(x_{[0,k-1]}) x_{[k,\infty)}. \]
\end{lemma}

\begin{proof}
It is enough to construct a homomorphism with this property, as it is then automatically an embedding by the third condition of safety. 

We will define a local rule $\chi : X \to 2^\Z$ that picks a set of coordinates where the rule can be safely applied. We need to ensure that $\chi(x) \subset \chi_U(x)$ for all $x$, so that we can apply each rule $\pi$ at this coordinate. The same local rule $\chi$ is used when defining each $\gamma(\pi)$, and we also ensure that $\chi(\gamma(\pi)(x)) = \chi(x)$ for all $x \in X$ and $\pi \in \Gamma$, so that we obtain a group homomorphism. We also have to make sure that the behavior on $Y_1$ is correct.

Let $|\Sigma| = n$. We pick $m, \ell \in \N$ with $\ell > n^h$ and $m \geq \ell + 2k + h$.\footnote{The reason we need $n^k$ is that since we allow $V$ to be any set of words in the definition, it is possible that a word that does not contain $V$ has only continuations containing $V$, and a padding of $n^h$ is sufficient to get around this. In our application $V$ consists of a single word of length one, and in this special case one could simplify the construction, and reduce the radius.} Let $k' = k-1, h' = h-1$.

We let $i \in \chi(x)$ if and only if
\[ \chi_U(x) \cap [i-m, i+m] = \{i\} \]
and
\[ \chi_V(x) \cap [i-\ell, i+\ell] \subset [i,i+k-h] = [i,i+k'-h']. \]
We can now define the map $\gamma(\pi)$: if $i \in \Z$ and $0 \leq j < k$, then
\[ i-j \in \chi(x) \implies \gamma(\pi)(x)_{i} = \pi(x_{[i-j,i-j+k']})_j \]
and $\gamma(\pi)(x)_i = x_i$ otherwise. This clearly defines a continuous shift-commuting function, and it is well-defined since $m \geq k$, and thus for each $i \in \Z$ the first defining rule can apply to at most one $j$.

With this definition, we already have the correct behavior on $Y_1$. To show that the rules $\gamma(\pi)$ are automorphisms and $\gamma$ is a group homomorphism, it is sufficient to prove that $\chi(\gamma(\pi)(x)) = \chi(x)$ for all $x$, as then $\gamma(\pi)(\gamma(\pi')(x)) = \gamma(\pi \circ \pi')(x)$ for all $x \in \Sigma^\Z$ and $\pi \in \Gamma$.

Suppose that $i \in \chi(x)$, that is, $\chi_U(x) \cap [i-m, i+m] = \{i\}$ and $\chi_V(x) \cap [i-\ell, i+\ell] \subset [i,i+k'-h']$. Then the symbols at coordinates $[i-m+k',i-1] \cup [i+k,i+m]$ do not change when $\gamma(\pi)$ is applied. Thus, if we had
\[ \chi_U(\gamma(\pi)(x)) \cap [i-m+k', i+m-k'] \neq \{i\}, \]
then we would necessarily have $\chi_U(\gamma(\pi)(x)) \cap [i-k',i+k'] \neq \{i\}$. Since $m-2k'-h' > \ell > n^h$, by a pigeonhole argument,
we could then find a point $y$ which is eventually periodic both left and right, $y_{[i-k', i+2k']} = x_{[i-k', i+2k']}$, and $\chi_V(y) \subset [i,i+k'-h']$, and the fact we obtain two occurrences of words in $U$ when changing the word $x_{[i,i+k']}$ to $\pi(x_{[i,i+k']})$ contradicts the $V$-niceness of $U$.

Thus, if $i \in \chi(x)$, we have $\chi_U(\gamma(\pi)(x)) \cap [i-m+k', i+m-k'] = \{i\}$. From this, the fact that coordinates in $[i-m+k',i-1] \cup [i+k,i+m]$ do not change when a rewrite happens at $i$, and $m - k' > k'$, we have $\chi_U(\gamma(\pi)(x)) = \chi_U(x)$ for all $x$.

We do not necessarily have $\chi_V(\gamma(\pi)(x)) = \chi_V(x)$, but if $i \in \chi(x)$, then since the coordinates $[i-m+k',i-1] \cup [i+k,i+m]$ do not change when $x_{[i,i+k-1]}$ is rewritten, after the rewrite we have
\[ \chi_V(\gamma(\pi)(x)) \cap [i-\ell, i+\ell] \subset \chi_V(\gamma(\pi)(x)) \cap [i-m+k', i+m] \subset [i,i+k'-h'], \]
so $i \in \chi(\gamma(\pi)(x))$.

Finally, if $i \notin \chi(x)$ but $\chi_U(x) \cap [i-m, i+m] = \{i\}$, then we have $j \in \chi_V(x) \cap [i-\ell, i+\ell]$ for some $j \notin [i,i+k'-h']$. We then also have $j \in \chi_V(\gamma(\pi)(x)) \cap [i-\ell, i+\ell]$, since $m - 2k' - h' \geq \ell$, implying $i \notin \chi(\gamma(\pi)(x))$, as $i \in \chi_U(x)$ prevents rewritings affecting $j \in \chi_V(\gamma(\pi)(x))$ (by a similar argument as above).
\end{proof}

Clearly being $V$-safe is a decidable condition for any $V \subset \Sigma^h$, but a precise characterization would presumably involve some intricate combinatorics on words. In our application, we have $X$ a full shift, $V = \{s\}$, and we will guarantee safety by assuming no words in $U$ contain $s$ at their border, and the following sufficient condition for a set of words to be safe is sufficient for our purposes.

\begin{lemma}
\label{lem:SufficientSafety}
Let $\Sigma = \{0,1,\ldots,n-1\}$, $A = \Sigma \setminus \{s\}$ and $s = n-1$. Let $U \subset A^k \Sigma^k A^k$ be a set of words such that every word in $U$ contains $s$, and if $u, v \in U$ and $|u|_s = |v|_s$, then the leftmost occurrence of $s$ is in the same coordinate in $u$ and $v$. Then $U$ is an $\{s\}$-safe set of words on the full shift $\Sigma^\Z$.
\end{lemma}

\begin{proof}
The first condition of the definition of safety is trivial.

For the second, let $V = \{s\}$ and suppose that $x$ satisfies $x_{[0,3k-1]} = u$, $\chi_U(x) = \{0\}$ and $\chi_V(x) \subset [0,3k-1]$. Let $y = x_{(-\infty, -1]}. v x_{[3k,\infty)}$ where $v \in U$. Then by the form of $U$, we in fact have $\chi_V(x) \subset [k,2k-1]$. If $\chi_U(y) \neq \{0\}$, then $y_{[i,i+3k-1]} = w \in U$ for some $i \in [-3k+1,3k-1] \setminus \{0\}$. Since only the middle third of each of $v$ and $w$ contains $s$s and $\chi_V(x) \subset [k,2k-1]$, it is easy to see that $|i| < k$, and thus $v$ and $w$ contain the same number of $s$s. This is a contradiction, since $i \neq 0$ so the leftmost $s$ is not in the same coordinate in these two words.

The third condition of the definition is satisfied by $..000.u000...$ by a similar argument.
\end{proof}


\begin{lemma}
\label{lem:SimulatedShift}
Let $\Sigma = \{0,1,\ldots,n-1\}$ and $s = n-1$. Let $Y_{1,s} \subset \Sigma^\Z$ be the set of configurations containing $s$ at the origin and having no other occurrence of $s$. Then there is an automorphism $f \in \Aut(\Sigma^\Z)$ with the property
\[ \forall x \in Y_{1,s}: f(x)_0 = x_1 \wedge f(x)_1 = x_0 \wedge \forall i \notin \{0,1\}: f(x)_i = x_i. \]
\end{lemma}

\begin{proof}
Let $A = \Sigma \setminus \{s\}$. Intuitively, the solution is that we first perform an involution that changes every occurrence of $s$ (that is suitably isolated) preceded by $a \in A$ into an occurrence of $sws$, where the distance between the $s$s codes $a$, and $w$ is taken from the configuration (so no information is lost). Another involution changes this to $sa$, so that composing the two involutions, we effectively move $s$ to the left or right, depending on the order of composition.

Pick $U_1 = \{A^{m+1} s A^{m-1}\}$ and $U_2 = \{A^m s A^k s A^{m-1-k} \;|\; k \in \{0,\ldots,n-2\}\}$. Observe that $U_1$ and $U_2$ are disjoint, both have $(n-1)^{2m}$ elements, and all words in these sets are of length $2m+1$. Let $U = U_1 \cup U_2$ and pick any bijection $\pi' : U_1 \to U_2$. For example, one can pick $\pi'(u a s v) = u s v_{[0,a-1]} s v_{[a,m-2]}$ for $u \in A^m, a \in A, v \in A^{m-1}$. This gives us an involution $\pi : U \to U$ defined by $\pi|_{U_1} = \pi'$ and $\pi|_{U_2} = (\pi')^{-1}$.

If we pick $m$ large enough, and such that $2m+1$ is divisible by $3$, then the assumptions of Lemma~\ref{lem:SufficientSafety} are satisfied, and there is a homomorphism $\gamma$ from the group of $0$-safe bijections (that is, all bijections) on $U$ to $\Aut(\Sigma^\Z)$ which satisfies the conclusion of Lemma~\ref{lem:TechnicalSafety}. Let $f_{\pi} = \gamma(\pi)$.

Pick now $V_1 = \{A^{m} s A^m\}$, $V_2 = U_2$ and $V = V_1 \cup V_2$ and construct similarly an involution $\tau : V \to V$, picking the bijection $\tau' : V_1 \to V_2$ so that $(\pi')^{-1} \circ \tau'(u s a v) = u a s v$ for all $a \in A, u \in A^{m}, v \in A^{m-1}$. Construct the automorphism $f_\tau$ as above, using the involution $\tau$.

It is now easy to verify that $f = f_\pi \circ f_\tau$ acts correctly on $Y_{1,s}$.
\end{proof}

Let $s \in \Sigma \setminus \{0\}$, and let $Y_{1,s}$ be as in the previous lemma. Let $A = \Sigma \setminus \{s\}$ and $Y_{0,s} = A^\Z$.


\begin{lemma}
\label{lem:AlmostTransitively}
Let $s \in \Sigma \setminus \{0\}$. 
Then there is a finitely generated subgroup $\mathcal{G}_{s}$ of $\Aut(\Sigma^\Z)$ that acts trivially on $Y_{0,s}$, and is $\infty$-orbit-transitive around $Y_{1,s}$.
\end{lemma}



\begin{proof}
Let $A = \Sigma \setminus \{s\}$. We define a map $\phi : G_0 \to \Aut(\Sigma^\Z)$ with the following properties:
\begin{itemize}
\item for all $g \in G_0$, $\phi(g)(x)_i = x_i$ unless $x$ contains $s$ at a bounded distance from $i$ (the bound depending on $g$), and
\item for $g \in G_0$ and for all $x \in A^{-\N}$, $y \in A^\N$ we have
\[ g(x.y) = x'.y' \implies \phi(g)(x.sy) = \sigma^{-\psi(g)}(x'.sy'), \]
\end{itemize}
where $\psi : G_0 \to \Z$ is the homomorphism from Lemma~\ref{lem:TMTransitive}.

From these properties it will follow that if $F$ is a finite generating set for $G_0$, then $\phi(F)$ generates a finitely generated subgroup of $\Aut(\Sigma^\Z)$ with the required properties: the first implies that $\phi(g)$ acts trivially on $Y_{0,s}$ and $\infty$-orbit-transitivity around $Y_{1,s}$ comes from the second, since the action of $[G_0,G_0]$ is $\infty$-transitive on $X'_0$ and $\psi(g) = 0$ for $g \in [G_0,G_0]$.

We remark that these properties also imply that if we let $Y$ be the (sofic) subshift obtained as the orbit closure of $Y_{1,s}$ under the shift, then the restriction $\phi(g)|_Y : Y \to Y$ is well-defined for all $g \in G_0$, and the map $g \mapsto \phi(g)|_Y$ is a group embedding from $G_0$ to $\Aut(Y)$. However, $\phi$ will not be a group embedding of $G_0$ to $\Aut(\Sigma^\Z)$, and could not be, since $\Aut(\Sigma^\Z)$ is residually finite while $G_0$ contains the infinite simple group $[G_0,G_0]$.

Let $G_0$ be the finitely generated group of homeomorphisms of $A^\Z$ defined in Definition~\ref{def:OB}. To each permutation $\pi : A^{[-r,r]} \to A^{[-r,r]}$, we associate the CA $f_{\pi,s}$ defined by rewriting words $usv$ where $u \in A^r, v \in A^{r}$ to $u'sv'$ where $u'v' = \pi(uv)$ and $|u'| = r$, whenever there are no other occurrences of the symbol $s$ are at distance less than $2r+3$ from the central $s$ in the pattern $usv$. Since occurrences of $s$ are neither removed nor introduced, this gives a well-defined automorphism.

Now, if $g : A^\Z \to A^\Z$ is a local homeomorphism with $\psi(g) = 0$, choose a presentation $g = P_\pi$ with $\pi : A^{[-r,r]} \to A^{[-r,r]}$ arbitrarily (for example, pick the smallest possible $r$) and let $\phi(g) = f_{\pi,s}$. It is clear that $\phi(g)$ then has the required properties for any $g \in G_0'$.

Let then $\phi(\sigma) = f_{\sigma,s}$, where $f_{\sigma,s}$ is the automorphism constructed in Lemma~\ref{lem:SimulatedShift}. Let $x \in A^{-\N}$, $a \in A$ and $y \in A^\N$. Plugging $\sigma(x.ay) = xa.y$ and $\psi(\sigma) = 1$ in the second requirement for $\phi$ above, we get
\[ \sigma(x.ay) = xa.y \implies \phi(\sigma)(x.say) = \sigma^{-1}(xa.sy) = x.asy, \]
and indeed $f_{\sigma,s}(x.say) = x.asy$.

Now, if $F$ is a finite set of permutations $\pi : \Sigma^N \to \Sigma^N$ such that $G_0$ is generated by $\sigma$ and $\{P_\pi \;|\; \pi \in F\}$, then
\[ \mathcal{G}_{s} = \langle f_{\pi,s}, f_{\sigma,s} \;|\; \pi \in F \rangle \]
is a finitely generated subgroup of $\Aut(\Sigma^\Z)$ with the required properties.
\end{proof}


\subsection{The reset system}

We also define the following somewhat peculiar monoid action that is helpful for book-keeping in the proof of our main theorem.

\begin{definition}
Let $M$ be the free (noncommutative) monoid generated by $(\Z \times \N)^k$.\footnote{In other words, $M$ is the set of words $\vec x_1 \vec x_2 \cdots \vec x_\ell$ where $\vec x_i \in (\Z \times \N)^k$ for all $i \in [1, \ell]$, with concatenation as the monoid operation.} 
The generator $\vec x \in (\Z \times \N)^k$ of $M$ acts on $(\Z \times \N)^k$ by
\[ (\vec x \cdot \vec y)_i = \left\{\begin{array}{ll}
(n_i, t_i) & \mbox{if } \vec y_i = (n_i, t_i+1) \mbox{ and} \\
x_i & \mbox{if } \vec y_i = (n, 0). \\
\end{array}\right. \]
The dynamical system $(M, (\Z \times \N)^k)$ is called the \emph{reset system}.
\end{definition}

One may think of the reset system as a dynamical system modeling a finite set of alarm clocks that, when they buzz, are reset to buzz at a later time. In addition to the counter $\N$, the alarm clocks carry a position in $\Z$ indicating \emph{where} they buzz next. 

\begin{lemma}
\label{lem:ResetIsTransitive}
The reset system is transitive. In particular, for all $\vec v \in (\Z \times \N)^k$ there exists $m \in M$ such that $m \cdot \vec v = (0,0)^k$.
\end{lemma}

\begin{proof}
Let $\vec u, \vec v \in (\Z \times \N)^k$ be arbitrary. Then for large enough $\ell$ we have $((0,0)^k)^\ell \cdot \vec v = (0,0)^k$. Thus $(\vec u \cdot ((0,0)^k)^\ell) \cdot \vec v = \vec u$, proving transitivity since $\vec u \cdot ((0,0)^k)^\ell \in M$ and $\vec u, \vec v$ were arbitrary.
\end{proof}

In the proof of the main theorem, the positions $\Z$ will be eventual positions of heads, and $\N$ will represent the times when heads first appear in those positions when a suitable particle rule $P$ is applied repeatedly. We can simulate the (highly non-reversible) action of the reset system by automorphisms, as we only need to act in a controlled way on a finite set of points.

We remark that in the proof of the main theorem, most of the real work happens when a (simulated) clock is snoozed, that is, when the second component of some $\vec y_i$ reaches zero and $(\vec x \cdot \vec y)_i = \vec x_i$. If one wants `efficient transitivity' (where the transporter has small word norm in the group) in the proof of the main theorem, one might prefer a more elegant proof of the previous lemma. With a bit of arithmetic, one can prove the lemma without snoozing any clock more than once. 




\section{Proof of the main result}

In this section, we prove the main result Theorem~\ref{thm:Main}. For this section, fix $X = \Sigma^\Z$ where $\Sigma = \{0,1,2,3\}^\Z$. Let $P$ be the particle rule with the decomposition $\Sigma \cong \{0,1\} \times \{0,1\}$ given by $n \mapsto (n\bmod 2, \lfloor n/2 \rfloor)$.

The $\infty$-orbit-transitive action will be given by the group
\[ \mathcal{G} = \langle \mathcal{G}_{3}, P, g_{(13)}, g_{(23)} \rangle \]
where 
$g_{(13)}$ is the symbol-permutation $(1 \; 3)$, $g_{(23)}$ the symbol-permutation $(2 \; 3)$, and the group $\mathcal{G}_3$ is that from Lemma~\ref{lem:AlmostTransitively}.

We call $1$ a \emph{particle} and $2$ a \emph{wall}. We call $3$ a \emph{head}, also understanding it as a particle on top of a wall. By a \emph{collision}, we refer to the intuitive concept of the appearance of a head due to the particle rule (or its inverse) moving a particle on top of a wall. Note that $P$ moves particles to the left and keeps walls fixed.

Though in the proofs we only use these numbers and names, we fix the following pictorial presentation as well:
\[ 0 = \Sy{0}, \; 1 = \Sy{1}, \; 2 = \Sy{2}, \;\mbox{and } 3 = \Sy{3}. \]

We refer to tuples $\vec{x} \in X^{(k)}_0$ as \emph{vectors}, and the points $\vec{x}_i, i \in [1,k]$ are called its \emph{components}.

\begin{definition}
A nonzero finite point $x \in X_0$ is
\begin{itemize}
\item \emph{prepregood} if it contains no heads and all the particles are to the left of all the walls,
\item \emph{pregood} if it is prepregood and contains a particle,
\item \emph{good} if it is pregood and contains a wall, and
\item \emph{great} if it contains exactly one head which is at $0$.
\end{itemize}
For any of these properties P, a vector is P if all its components are.
\end{definition}

Note that great points are just the points in $Y_{1,3}$. Note that prepregood points are pregood and pregood points are good, but good points are never great and great points are never good.


\subsection{Outline and example}

Suppose $k = 3$, and begin with the following vector $\vec x$ of configurations:

\vspace*{0.2cm}
$\vec x =$ \Sy{0,0,0,0,0,7,0,0,0,0,0} \Sy{0,0,0,0,2,4,1,0,0,0,0} \Sy{0,0,0,0,0,6,2,0,0,0,0}
\vspace*{0.1cm}

Since we are dealing with a group action, it is enough to transform $\vec x$ into a fixed vector $\vec y$, as the inverse steps of such transformations can be used to transform $\vec y$ into any other vector. We use the great vector

\vspace*{0.2cm}
$\vec y = $ \Sy{0,0,0,0,0,7,1,0,0,0,0} \Sy{0,0,0,0,0,7,0,1,0,0,0} \Sy{0,0,0,0,0,7,0,0,1,0,0}
\vspace*{0.1cm}

The outline of the proof is that we first make $\vec x$ prepregood, then pregood, then good, by a few applications of the particle rule and symbol-permutations. We then simulate the reset system using Lemma~\ref{lem:AlmostTransitively} to make $\vec x$ great. A final application of Lemma~\ref{lem:AlmostTransitively} finishes the proof, turning $\vec x$ into $\vec y$.

Figure~\ref{fig:Example} shows how to do this, using Lemma~\ref{lem:AlmostTransitively} as a black box. We have taken some artistic liberties in Figure~\ref{fig:Example}; it does not follow the proof to the letter. Intuitively, Lemma~\ref{lem:AlmostTransitively} allows us to do \emph{anything} around any individual head, as long as the relative movement of every head is the same. 

\begin{figure}
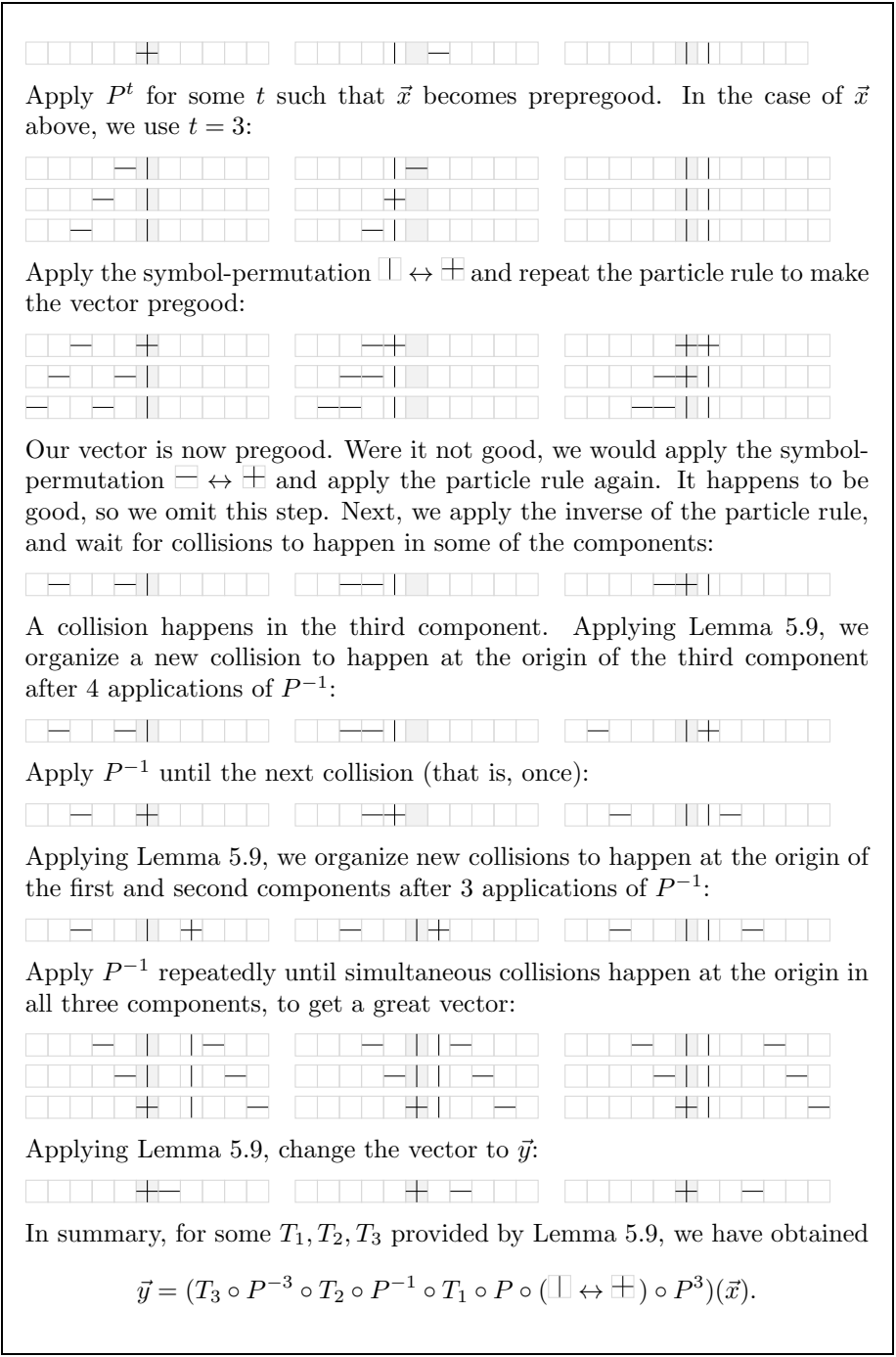

\begin{framed}
\vspace*{0.2cm}
\Sy{0,0,0,0,0,7,0,0,0,0,0} \Sy{0,0,0,0,2,4,1,0,0,0,0} \Sy{0,0,0,0,0,6,2,0,0,0,0}
\vspace*{0.1cm}

Apply $P^t$ for some $t$ such that $\vec x$ becomes prepregood. In the case of $\vec x$ above, we use $t = 3$:



\vspace*{0.2cm}
\Sy{0,0,0,0,1,6,0,0,0,0,0} \Sy{0,0,0,0,2,5,0,0,0,0,0} \Sy{0,0,0,0,0,6,2,0,0,0,0,0}

\Sy{0,0,0,1,0,6,0,0,0,0,0} \Sy{0,0,0,0,3,4,0,0,0,0,0} \Sy{0,0,0,0,0,6,2,0,0,0,0,0}

\Sy{0,0,1,0,0,6,0,0,0,0,0} \Sy{0,0,0,1,2,4,0,0,0,0,0} \Sy{0,0,0,0,0,6,2,0,0,0,0,0}
\vspace*{0.1cm}

Apply the symbol-permutation $\Sy{2} \leftrightarrow \Sy{3}$ and repeat the particle rule to make the vector pregood:

\vspace*{0.2cm}
\Sy{0,0,1,0,0,7,0,0,0,0,0} \Sy{0,0,0,1,3,4,0,0,0,0,0} \Sy{0,0,0,0,0,7,3,0,0,0,0,0}

\Sy{0,1,0,0,1,6,0,0,0,0,0} \Sy{0,0,1,1,2,4,0,0,0,0,0} \Sy{0,0,0,0,1,7,2,0,0,0,0,0}

\Sy{1,0,0,1,0,6,0,0,0,0,0} \Sy{0,1,1,0,2,4,0,0,0,0,0} \Sy{0,0,0,1,1,6,2,0,0,0,0,0}
\vspace*{0.1cm}

Our vector is now pregood. Were it not good, we would apply the symbol-permutation $\Sy{1} \leftrightarrow \Sy{3}$ and apply the particle rule again. It happens to be good, so we omit this step. Next, we apply the inverse of the particle rule, and wait for collisions to happen in some of the components:

\vspace*{0.2cm}
\Sy{0,1,0,0,1,6,0,0,0,0,0} \Sy{0,0,1,1,2,4,0,0,0,0,0} \Sy{0,0,0,0,1,7,2,0,0,0,0,0}
\vspace*{0.1cm}

A collision happens in the third component. Applying Lemma~\ref{lem:AlmostTransitively}, we organize a new collision to happen at the origin of the third component after $4$ applications of $P^{-1}$:

\vspace*{0.2cm}
\Sy{0,1,0,0,1,6,0,0,0,0,0} \Sy{0,0,1,1,2,4,0,0,0,0,0} \Sy{0,1,0,0,0,6,3,0,0,0,0,0}
\vspace*{0.1cm}

Apply $P^{-1}$ until the next collision (that is, once):

\vspace*{0.2cm}
\Sy{0,0,1,0,0,7,0,0,0,0,0} \Sy{0,0,0,1,3,4,0,0,0,0,0} \Sy{0,0,1,0,0,6,2,1,0,0,0,0}
\vspace*{0.1cm}

Applying Lemma~\ref{lem:AlmostTransitively}, we organize new collisions to happen at the origin of the first and second components after $3$ applications of $P^{-1}$:


\vspace*{0.2cm}
\Sy{0,0,1,0,0,6,0,3,0,0,0} \Sy{0,0,1,0,0,6,3,0,0,0,0} \Sy{0,0,1,0,0,6,2,0,1,0,0,0}
\vspace*{0.1cm}

Apply $P^{-1}$ repeatedly until simultaneous collisions happen at the origin in all three components, to get a great vector:

\vspace*{0.2cm}
\Sy{0,0,0,1,0,6,0,2,1,0,0} \Sy{0,0,0,1,0,6,2,1,0,0,0} \Sy{0,0,0,1,0,6,2,0,0,1,0,0}

\Sy{0,0,0,0,1,6,0,2,0,1,0} \Sy{0,0,0,0,1,6,2,0,1,0,0} \Sy{0,0,0,0,1,6,2,0,0,0,1,0}

\Sy{0,0,0,0,0,7,0,2,0,0,1} \Sy{0,0,0,0,0,7,2,0,0,1,0} \Sy{0,0,0,0,0,7,2,0,0,0,0,1}
\vspace*{0.1cm}

Applying Lemma~\ref{lem:AlmostTransitively}, change the vector to $\vec y$:

\vspace*{0.2cm}
\Sy{0,0,0,0,0,7,1,0,0,0,0} \Sy{0,0,0,0,0,7,0,1,0,0,0} \Sy{0,0,0,0,0,7,0,0,1,0,0,0}
\vspace*{0.1cm}

In summary, for some $T_1, T_2, T_3$ provided by Lemma~\ref{lem:AlmostTransitively}, we have obtained
\[ \vec y = (T_3 \circ P^{-3} \circ T_2 \circ P^{-1} \circ T_1 \circ P \circ (\Sy{2} \leftrightarrow \Sy{3}\,) \circ P^3)(\vec x). \]
\end{framed}
\caption{How to get from $\vec x$ to $\vec y$, using Lemma~\ref{lem:AlmostTransitively} as a black box. The shaded square denotes the origin.}
\label{fig:Example}
\end{figure}

\subsection{Making vectors good}

The first step in the proof is showing that we can always make our vector of configurations good. This is done in three steps: we first make the vector prepregood, then pregood, and then good.

\begin{lemma}
\label{lem:Good}
Let $\vec{x} \in X^{(k)}_0$. Then there is an element $g \in \langle g_{(13)}, g_{(23)}, P \rangle$ such that $g \vec{x}$ is good.
\end{lemma}

\begin{proof}
Applying $P$ moves particles to the left while keeping walls still. Clearly after some amount of steps of $P$ on $\vec x$, there are no heads and all particles are to the left of all the walls in all components of $\vec x$, so $\vec x$ can be transformed to a prepregood vector $\vec x^2$.

Now, apply $g_{(23)}$ to the prepregood vector $\vec x^2$, turning walls into heads, and repeatedly apply $P$ until all particles are again to the left of all the walls. Note that after this, every component of $\vec x^2$ contains a particle: if a component $x = \vec x_i^2$ contains a particle before applying $g_{(23)}$, it still does after applying it and repeating $P$. If it contains a wall, then a head is introduced, giving a particle after the application of $P$. Note that $x$ must contain either a particle or a wall, since $x \in X_0$ and $0^\Z \notin X_0$. We obtain that $\vec x^2$ can be transformed to a pregood vector $\vec x^3$.

By the same argument, by applying $g_{(13)}$ and then repeating $P$, we eventually turn $\vec x^3$ into a good vector $\vec x^4$.
\end{proof}

\subsection{Making vectors great}

We give a precise connection between the reset system and our automorphism action. For this we need to measure when and where our points are great, and we give the following definition: A point $x \in \Sigma^\Z$ is \emph{clock-like} if there exists $(a,t) \in \Z \times \N$ such that $\sigma^a(P^{-t-1}(x))$ is great (in particular $P^{-t-1}(x)_a = s$) 
and for $0 \leq t' < t$, $P^{-t'-1}(x)$ does not contain a collision. Note that we do not require that $x$ not contain a head (it typically will contain one in our application). A vector $\vec x$ is \emph{clock-like} if all its components are. Write $C \subset X_0$ for the set of finite clock-like points. 

Define $\phi : C^{(k)} \to (\Z \times \N)^k$ by associating to a clock-like point $x$ the pair $(a,t)$ as defined above. Note that the time value $t$ in the reset system corresponding to a component in a vector $\vec x$ is one less than the actual collision time when applying $P^{-1}$.

The following lemma gives a natural correspondence between the action of automorphisms on $C^{(k)}$, and the action of the reset system with $k$ clocks.\footnote{We are essentially saying that $\phi$ is a `homomorphism', which could be formalized in terms of abstract rewriting systems or categories.}


\begin{lemma}
\label{lem:ARSMorphism}
Let $\vec x \in C^{(k)}$ and $\vec v \in (\Z \times \N)^k$. Then there exists 
$g \in \mathcal{G}_{3}$
such that $g P^{-1} (\vec x) \in C^{(k)}$ and $\phi(g P^{-1} (\vec x)) = \vec v$ if and only if there exists $m \in (\Z \times \N)^k$ such that $m \cdot \phi(\vec x) = \vec v$.
\end{lemma}

\begin{proof}
We show that applying an element of the form $g P^{-1}$ where $g \in \mathcal{G}_3$ to a vector $\vec x \in C^{(k)}$ either takes $\vec x$ out of $C^{(k)}$ or corresponds to a translation in the reset system: Consider the application of some element $g P^{-1}$ to $\vec x$ such that $g P^{-1}(\vec x) \in C^{(k)}$. If $\phi(\vec x)_i = (a,t)$ with $t > 0$, we necessarily have $\phi(g P^{-1}(\vec x))_i = (a,t-1)$, since no collision happens when $P^{-1}$ is applied, and $g$ behaves as identity on points without heads. We simply choose $m \in (\Z \times \N)^k$ so that the action is correct on other components.

We now show a converse, namely that maps of the form $g P^{-1}$ implement all possible transitions in the reset system. More precisely, that for all $\vec x \in C^{(k)}$ and $m \in (\Z \times \N)^k$, there is an element $g \in \mathcal{G}_3$ such that $\phi(g P^{-1}(\vec x)) = m \cdot \phi(\vec x)$. For this, let $I \subset [1,k]$ be the set of those $i \in [1,k]$ such that there is a great point in the $\sigma$-orbit of $x = P^{-1}(\vec x_i)$. We need to find an element of $\mathcal{G}_3$ that resets the next buzz times and positions of these clocks according to the corresponding coordinates of $m_i$. Let $m_i = (n_i,t_i)$ for all $i$.

The fact that the points come from different $\sigma$-orbits means that we can really do any transformation we like using Lemma~\ref{lem:AlmostTransitively}, as long as the relative positions of the heads do not change. However, since Lemma~\ref{lem:AlmostTransitively} cannot get rid of heads, and heads dissolve into a wall and a particle when $P^{-1}$ is applied, we need to be careful so that the particle they introduce does not lead to sporadic collisions. Thus, we make the heads the rightmost symbols of the support, so that particles simply escape to infinity when we start applying $P^{-1}$.

We describe this in more detail. Note that automorphisms of the form $f_{\pi,3}$ and $f_{\sigma,3}$ act trivially on $Y_{0,3}$, and thus only modify components $P^{-1}(\vec x)_i$ where $i \in I$, so we can ignore coordinates $i \notin I$: they are already handled correctly, as $P^{-1}$ decrements the right components of their $\phi$-images.

On the components $P^{-1}(\vec x)_i$ where $i \in I$, first apply $f_{\sigma,3}^k$ for some large $k$ so that all heads are at least two steps to the right of the values $n_i$. Let the head on $\vec y_i = f_{\sigma,3}^k(P^{-1}(\vec x)_i)$ be in coordinate $k_i$ for all $i \in I$. Note that the points $\vec y_i$, $i \in I_1$, are from different $\sigma$-orbits by the definition of $C^{(k)}$, and since automorphisms preserve orbits. Now, use the $\infty$-transitivity of the action of $f_{\pi,3}$ and $f_{\sigma,3}$ around $Y_{1,3}$ given by Lemma~\ref{lem:AlmostTransitively} to, for all $i \in I_1$, transform $\vec y_i$ to the vector $\vec z_i$ where
\[ (\vec z_i)_j = \left\{\begin{array}{ll}
3, & \mbox{if } j = k_i, \\
2, & \mbox{if } j = n_i, \\
1, & \mbox{if } j \in \{n_i - t_i - 1, k_i + i + 1 \}, \mbox{ and} \\
0, & \mbox{otherwise.} \end{array}\right.
\]
Here, the particle in coordinate $n_i - t_i - 1$ and the wall at $n_i$ generate the collision at $n_i$ after $t_i+1$ steps of applying $P^{-1}$. The particle at $k_i + i + 1$ is included so that we are sure that these points are from distinct orbits.\footnote{Omitting this particle was the main `artistic liberty' taken in Figure~\ref{fig:Example}.}

Let $f \in \mathcal{G}_3$ be the transformation mapping $\vec y_i \to \vec z_i$. Then
\[ \phi(f \circ f_{\sigma,3}^k \circ P^{-1}(\vec x))_i = \phi(P^{-1}(\vec x_i)) = m_i \cdot \phi(\vec x_i) \]
for $i \notin I$ and
\[ \phi(f \circ f_{\sigma,3}^k \circ P^{-1}(\vec x))_i = \phi(\vec z_i) = (n_i, t_i) = m_i \cdot \phi(\vec x_i) \]
for $i \in I$, which concludes the proof.
\end{proof}


\begin{lemma}
\label{lem:Great}
If $\vec x$ is good, there exists
$ g \in \mathcal{G}' = \langle \mathcal{G}_3, P \rangle \leq \mathcal{G} $
such that $g \vec x$ is great.
\end{lemma}

\begin{proof}
It is clear that every good vector is clock-like. The reset system is transitive, so by the previous lemma for every $\vec x \in C^{(k)}$ there is a group element $g \in \mathcal{G}'$ such that $g \vec x = \vec x' \in C^{(k)}$ and $\phi(\vec x') = (0,0)^k$. Then $P^{-1}(\vec x')$ is great. 
\end{proof}

\subsection{Proof of $\infty$-orbit-transitivity}

Let $\vec y$ be the vector where the support of $\vec y_i$ is $\{0, i\}$, $(\vec y_i)_0 = 3$ and $(\vec y_i)_i = 1$.

\begin{lemma}
\label{lem:Final}
If $\vec x$ is great, then there exists $g \in \mathcal{G}_{3}$ such that $g \vec x = \vec y$.
\end{lemma}

\begin{proof}
This is a direct corollary of Lemma~\ref{lem:AlmostTransitively}.
\end{proof}

We can now prove our main theorem.

\begin{theorem}
\label{thm:MainProof}
Let $X = \{0,1,2,3\}^\Z$. Then
\[ \mathcal{G} = \langle \mathcal{G}_{3}, P, g_{(13)}, g_{(23)} \rangle \leq \Aut(X) \]
acts $\infty$-orbit-transitively on $X_0$.
\end{theorem}

\begin{proof}
Let $\vec x \in X_0^{(k)}$ be arbitrary. Then
\begin{itemize}
\item by Lemma~\ref{lem:Good} there exists $g \in \mathcal{G}$ such that $\vec x' = g \vec{x}$ is good,
\item by Lemma~\ref{lem:Great}, there exists $g' \in \mathcal{G}$ such that $\vec x'' = g' \vec x'$ is great, and
\item by Lemma~\ref{lem:Final}, there exists $g'' \in \mathcal{G}$ such that $g'' \vec x'' = \vec y$.
\end{itemize}

Since $\vec x$ and $k \in \N$ were arbitrary, and $\vec y$ is a function of $k$ only, this proves $\infty$-orbit-transitivity.
\end{proof}

\section{Notes}
\label{sec:Notes}




\subsection{The number of generators}

We have made little effort to minimize the generating set. The group $\mathcal{G}_{3}$ requires at most 5 generators, as one can extract from the fact that four gates are enough to generate the ternary reversible clone \cite{Bo15}. Thus, we need at most $8$ generators for the group $\mathcal{G}$. It is difficult to make the set essentially smaller without a new idea.

It is theoretically possible to find smaller generating sets by computer search -- any set of automorphisms generating our generators will naturally satisfy the conclusion of Theorem~\ref{thm:Main}. We have not attempted this. The radius of our generators is small, and within reach of complete enumeration of rules, apart from $f_{\sigma,3}$, whose neighborhood size (if obtained abstractly from our proof) is big.

\subsection{Abelian actions are not $2$-orbit-transitive}

As mentioned in the introduction, at least two automorphisms are needed for $\infty$-orbit-transitivity. We obtain this from the following more general proposition.

\begin{proposition}
\label{prop:AbelianNotOrbitTransitive}
No action of a finitely generated abelian group $G$ on an $H$-set $X$ with infinitely many orbits is $2$-orbit-transitive.
\end{proposition}

\begin{proof}
Assume $G$ acts $2$-orbit transitively on $X$. We may assume $G$ acts freely: If $G_0$ is the stabilizer of some $x$, then it is the stabilizer of every element of $X$, since $G$ is abelian and its action is transitive. Thus, the group $G/G_0$ acts $2$-orbit-transitively on $X$ with trivial stabilizers.


Let $x, y \in X$ be representatives of two distinct $H$-orbits. Then there exists $g$ such that $(x+g, y+g) = (y, x)$, so that $x+2g = x$. It follows that the torsion subgroup of $G$ acts transitively on $X$.

Since every subgroup of a finitely generated abelian group is finitely generated, the torsion subgroup of $G$ is finitely generated, and thus finite. From the freeness of the action, it follows that $G$ is finite, contradicting the fact that $X$ has infinitely many orbits.
\end{proof}

\subsection{Finite groups of automorphisms have big centralizer}

In \cite{SaTo14b}, it was shown that on every large enough full shift, there is a finite set of cellular automata having a trivial centralizer. A crucial difference between the monoid and group case is that the cellular automata of \cite{SaTo14b} in fact generate a finite monoid, while in the case of group actions, any group of cellular automata having a trivial centralizer has to be infinite, according to the following theorem. We follow \cite{Sa16a}, but this is a modest elaboration of \cite{KiRo90}.

\CentralizerOfFiniteGroup

\begin{proof}
Without loss of generality, we may assume $\Gamma = \{0,1\}$, since the automorphism groups of full shifts embed into each other \cite{KiRo90}. (Alternatively, one can directly modify the proof below.)

Let $X \subset \Sigma^\Z$. Without loss of generality, we can assume $G$ acts by cellwise symbol-permutations. Namely, by the conjugacy mapping $x \in X$ to $y \in (\Sigma^G)^\Z$ by $(y_i)_g  = g(x)_i$, the action of $G$ becomes a cellwise symbol-permutation. We may of course also assume the action is faithful.

Since $X$ is mixing, there exists a point $x \in X$ on which the action of $G$ is free ($|Gx| = |G|$) and thus there exists a word $w \in X$ such that whenever $w$ appears in a point, the action is free on that point. Using a higher block presentation \cite{LiMa95},\footnote{Note that this preserves the property of $G$ acting cellwise.} we can thus further assume that there is a symbol $a \in \Sigma$ such that $|A| = |G|$ where $A = Ga$ (where by $Ga$ we mean the orbit of $a$ under the symbol permutations by which $G$ acts). We may further assume that the subshift $Y$ of points $x \in X$ where no element of $A$ appears has positive entropy, by possibly passing to yet another higher block presentation. We may assume that the forbidden patterns defining $X$ are of  size $2$, again by passing to a higher block presentation.

Since $X$ is mixing and $Y$ has positive entropy, there exist $v, v' \in \Sigma^*$ such that for any large enough $m$ there are many words $a v w v' a$ and $w \in (\Sigma \setminus A)^m$. To see this, observe that by mixing and positive entropy of $Y$, for some symbols $b, c \in A$ and for all large enough $m = |w|$ there are many words of the form $b w c$ which occur in $X$, such that $w$ does not contain symbols in $A$. By mixing of $X$, there are then many words of the form $a u b w c u' a$ where $u$ and $u'$ are two fixed words. By picking a large $m$, performing yet another block coding and replacing $a$ by $cu'aub$, we may assume there are at least $4$ words occurring in $X$ of the form $awa$ where $|w| = m$ and $w$ contains no symbol in $A$, and enumerate them as $\{a w_{(0,0)} a, aw_{(0,1)}a, aw_{(1,0)}a, aw_{(1,1)} a\}$. Let $B = \{(0,0),(0,1),(1,0),(1,1)\}$.

Now, all words of the form $aw_{i_1}aw_{i_2}a \cdots aw_{i_\ell}a$ for $i_j \in B$ occur in $X$. Write $\psi$ for the obvious bijection between words of this form and words over the alphabet $B$. A word of length $\ell$ over the alphabet $B$ represents an element of $\{0,1\}^{2\ell}$ through the bijection $\pi : B^\ell \to \{0,1\}^{2\ell}$ defined by
\[ \pi(u)_i = \left\{\begin{array}{ll}
c & \mbox{if } 0 \leq i < \ell, u_i = (c,d), \\
d & \mbox{if } \ell \leq i < 2\ell, u_{2\ell-i-1} = (c,d).
\end{array}\right. \]

If $f \in \Aut(\{0,1\}^\Z)$, abusing notation we define an action of $f$ on words of the form $v = aw_{i_1}aw_{i_2}a \cdots aw_{i_\ell}a$ by
\[ \phi(f)(v) = \psi^{-1}(\pi^{-1}(f(\pi(\psi(v))^\Z)_{[0,2\ell-1]})), \]
where by $\pi(\psi(v))^\Z$ we denote the unique point $y$ with $\sigma$-period $2\ell$ and $y_{[0,2\ell-1]} = \pi(\psi(v))$. Intuitively, this means applying the local rule of $f$ `on the top track' of the word $\psi(v)$, the left-right reverse of its local rule on the bottom track, and turning around corners in a natural way, to simulate an action of $f$ on a periodic point of period $2\ell$. Then $\phi$ is a homomorphism from $\Aut(\{0,1\}^\Z)$ to the permutation group of $\{0,1\}^{2\ell}$.

We now define $\phi(f)$ to $X$ by applying $\phi(f)$ as defined above to maximal subwords of the form $aw_{i_1}aw_{i_2}a \cdots aw_{i_\ell}a$ (extending to the infinite case in the only possible consistent way), obtaining a homomorphism from $\Aut(\{0,1\}^\Z)$ to $\Aut(X)$.

The images of this homomorphism do not commute with the action of $G$. To make them commute, note that the set positions of a point $x \in X$ where a symbol from $A$ appears is preserved under the action of $G$. In particular, words of the form $b u_{i_1}b u_{i_2} b \cdots b u_{i_\ell} b$ that are in the $G$-orbit of a word $aw_{i_1}aw_{i_2}a \cdots aw_{i_\ell}a$ have no nontrivial overlaps for distinct $b$, $\ell$ and $i_j$. Thus, we modify the definition of the maps $\phi(f)$, and define them on such words $b u_{i_1}b u_{i_2} b \cdots b u_{i_\ell} b$ by conjugating by the unique $g \in G$ such that $gb = a$.
\end{proof}

Note that centralizers of finite groups of automorphisms are precisely the automorphisms of $G$-SFTs \cite{BoMeEi15} where $G$ is finite and acts faithfully (though not necessarily freely), if a $G$-SFT is considered as a $(G \times \Z)$-dynamical system.

One may wonder if there is a compactness result about Ryan's theorem stating that any set of automorphisms having trivial centralizer in fact has a finite subset with this property. This is doubtful, as if there exists a locally finite group acting by automorphisms and having a trivial centralizer, then there is no such compactness result by the previous theorem. However, we do not attempt to construct such an action here.

\subsection{Full shifts on other groups}

In the introduction we emphasized the mixing SFTs as a natural setting for the result. Another natural setting are full shifts on more general groups \cite{CeCo10}. There are technical difficulties with generalizing the result to torsion groups (already in the abelian case). Namely, finite points can have nontrivial stabilizers when $G$ has torsion elements:

\begin{proposition}
Let $G$ be a countable group and $\Sigma \ni 0$ a nontrivial alphabet. Then there is a nonzero finite point $x \in \Sigma^G$ with a nontrivial stabilizer if and only if $G$ has a torsion element. In particular the automorphism group of a full shift on a non-torsion-free group never acts transitively on nonzero finite points.
\end{proposition}

\begin{proof}
For the first claim, let $x$ be a finite nonzero point with a nontrivial stabilizer, and suppose $g \neq 1$ and $g x = x$. Let $A \subset G$ be the support of $x$, so that the support of $gx$ is $gA = A$. Then $g$ acts bijectively on the finite set $A \subset G$ by $a \mapsto g \cdot a$, so that $g^n \cdot a = a$ for some $n$ and $a \in A$. It follows that $g^n = 1$, and $g$ is a torsion element. Conversely, if $g^n = 1$, define $x \in \Sigma^G$ by $x_h = 1 \iff h \in \langle g \rangle$. Then $g x = x$.

For the second claim, defining $y_h = 1 \iff h = 1$ we have $g y \neq y$ for any $g \neq 1$. If $g x = x$, then no automorphism can take $x$ to $y$.
\end{proof}


As for torsion-free groups, the proof of the main theorem seems to generalize directly to $\Z^d$ and the same idea seems to work more generally on biorderable groups, but the general case is harder, since it is hard to make heads appear in a controlled way without knowing the geometry of the group.

\begin{question}
Let $\Sigma = \{0,1,2,3\}$, and let $G$ be finitely generated and torsion-free. Is there a finitely generated subgroup of $\Aut(\Sigma^G)$ that acts $\infty$-transitively on the nonzero finite points?
\end{question}



In the case of nontrivial torsion, one could of course require only transitivity between finite points with the same stabilizer, or alternatively only consider finite points with a trivial stabilizer, to get a definition for an action to be `as transitive as possible', and ask similar questions. 

\subsection{Finite witnesses for not being a shift map}


The following lemma was needed in the introduction.

\begin{lemma}
\label{lem:gxNotAShift}
Let $X$ be a pointed mixing SFT and $g \in \Aut(X) \setminus \langle \sigma \rangle$. Then there is a nonzero finite point $x \in X$ such that $g(x) \notin \mathcal{O}(x)$.
\end{lemma}

\begin{proof}
We may suppose $g(0^\Z) = 0^\Z$, as the claim is trivial otherwise. Let $r$ be the radius of $g$, that is, the minimal $r \in \N$ such that $g(x)_i$ is uniquely determined by $x_{[i-r,i+r]}$. Consider a point where a pattern $0^m u 0^m$ occurs, where $m \geq 2r$ and $u$ is nonzero. Since $g$ is an automorphism, there must be a nonzero symbol at distance at most $r$ from the support of $u$, since otherwise $...000.u000...$ is mapped to $0^\Z$. It follows that for any point $x = \ldots 000 . u 0^m v 000 \ldots$, if $g(x) = \sigma^k(x)$, then necessarily $|k| \leq r + |u|$. If $g \notin \langle \sigma \rangle$, then there is a point $y \in X$ such that $g(y) \notin \{\sigma^j(y) \;|\; j \in [-r-|u|, r+|u|]\}$. Picking $x = \ldots 000 . u 0^m v 000 \ldots$ so that $v$ contains a large central pattern of $y$, we obtain that $g(x)$ cannot be in the $\sigma$-orbit of $x$.
\end{proof}

\subsection{Basic facts about $k$}

We defined the invariant $k(X)$ for a subshift $X$ in Definition~\ref{def:k}. Here we make some easy observations and ask some questions.

First, note that the automorphism group is always countable, so that we need not distinguish between infinities, justifying the a priori sloppy notation $k(X) = \infty$. The case $k(X) = 1$ cannot happen for any subshift $X$: if we can take $F = \{f\}$, then since $f \circ f = f \circ f$, $f$ is a power of the shift. But then $F= \{\}$ suffices since every automorphism commutes with the shift. The case $k(X) = 0$ is equivalent to shift maps being the only automorphisms, which is common for minimal subshifts \cite{Ol13,CoYa14,CyKr16a,DoDuMaPe16}. It can also happen for trivial reasons, and in particular it is true for the minimal SFTs, that is, subshifts consisting of the orbit of one periodic point. The case $k(X) = \bot$ is equivalent to the center of the automorphism group containing an automorphism that is not a power of a shift, which is again common for minimal systems \cite{BoLiRu88,Sa14d,HoPa89,DoDuMaPe16} and can also happen for SFTs for trivial reasons, as it happens in the two-point subshift $X = \{0^\Z, 1^\Z\}$.  
By Ryan's theorem, $k(X) \neq \bot$ for all mixing SFTs. For the finite SFT $X_n = \{0^\Z,1^\Z,\ldots,(n-1)^\Z\}$ where $n \geq 3$ (in other words, the identity function on $\{0,\ldots,n-1\}$), we have $k(X_n) = 2$ since symmetric groups are generated by two permutations and have trivial center. (For the general case of a finite subshift see Proposition~\ref{prop:FiniteSubshift}.)

\begin{question}
What is $k(X)$ for an infinite mixing SFT $X$? Is it computable? Is it always finite? Is it always 2? Is $k(\Aut(\{0,1\}^\Z)) = k(\Aut(\{0,1,2\}^\Z))$?
\end{question}

In the proof of Theorem~\ref{thm:Ryan}, we needed symbol-permutations in addition to the generators of $\mathcal{G}$ to ensure $g$ fixes $0^\Z$. Two symbol-permutations suffice for this, so $k(\{0,1,2,3\}^\Z) \leq 10$.

For a group $G$, write $k(G)$ for the cardinality of the smallest subset $F$ of $G$ that has trivial centralizer, that is, $\bigcap_{g \in F} C_G(g) = \{1_G\}$ if such $F$ exists, and $k(G) = \bot$ otherwise (that is, when $G$ has nontrivial center).

\begin{lemma}
For every finite group $G$, there exists a minimal subshift $X_G$ with $k(X_G) = k(G)$.
\end{lemma}

\begin{proof}
Let $X_G$ be a minimal subshift with $\Aut(X_G) \cong G \times \Z$ where the $\Z$ corresponds to powers of the shift. Such a subshift exists by \cite{HoPa89,DoDuMaPe16}. Then if $F \subset G$ is such that $\bigcap_{g \in F} C_G(g) = \{1_G\}$, there is clearly a corresponding set of automorphisms of size $F$, so that $k(X_G) \leq k(G)$. If $F \subset \Aut(X_G)$ is a set of automorphisms, then these automorphisms are of the form $f_i \circ \sigma^{k_i}$ for some $k_i \in \Z$. Then $f \circ \sigma^k$ commutes with all of $f_i \circ \sigma^{k_i}$ if and only if $f$ commutes with all of $f_i$. But this happens if and only if the $f_i$ correspond to a subset of $G$ with trivial centralizer.
\end{proof}


I do not know what values $k(G)$ occur for finite groups $G$. If every value occurs, then it occurs also as $k(X)$ for a minimal subshift $X$ by the lemma.

In the category of countable subshifts, we can have $k(X) = \infty$:

\begin{example}
Let $X \subset \{0,1\}^\Z$ be the (Cantor-Bendixson rank 3 countable sofic) subshift where every word with at least $3$ symbols $1$ is forbidden. Then up to a shift, every automorphism is simply a finite-support permutation of $\N$ (as it permutes the distance of two $1$s). Thus if $F$ is any finite set of automorphisms, we can assume it to come from a finite set of finite-support permutations $P$ of $\N$. Any permutation $\pi$ that is the identity on the support of the permutations in $P$ commutes with them, and automorphisms corresponding to such $\pi$ need not be shift maps. On the other hand, the automorphism group is easily seen to have only shifts in its center, so $k(X) = \infty$.
\end{example}

In the category of finite subshifts (which are just permutations on a finite set), it is easy to solve the possible values of $k(X)$ by a bit of case analysis.

If $X$ is a nonempty finite subshift, then the shift map $\sigma : X \to X$ is a permutation. Let $c_i$ be the number of cycles of length $i$ in the cycle decomposition of $\sigma$. Let $P = \{i \;|\; c_i \geq 1\}$ and $n = |P \cap [2,\infty)|$. If $P \cap [2,\infty) = \{\ell\}$ let $c = c_{\ell}$ and otherwise let $c = \bot$.

\begin{proposition}
\label{prop:FiniteSubshift}
Let $X$ be a nonempty finite subshift, and let $c_i, P, n, c$ be as above. Then $k(X) \in \{\bot,0,2\}$ and
\begin{align*}
k(X) = \bot & \iff c_1 = 2 \vee n \geq 2, \\
k(X) = 0    & \iff c_1 \in \{0,1\} \wedge (n = 0 \vee (n = 1 \wedge c = 1)), \\
k(X) = 2    & \iff (c_1 \geq 3 \wedge n < 2) \vee (c_1 \neq 2 \wedge n = 1 \wedge c \geq 2).
\end{align*}
\end{proposition}

\begin{proof}
We perform a case analysis based on the values of $c_1$, $n$ and $c$. The Boolean formulas in the statement are obtained by picking the corresponding rows.
\begin{center}
\begin{tabular}{ |c|c|c||c|l| }
  \hline
  $c_1$    & $n$      & $c$      & $k(X)$ & reason(s)          \\
  \hline
  $0$      & $0$      & $\bot$   & $0$    & $X = \emptyset$ \\
  $0$      & $1$      & $1$      & $0$    & (c)             \\
  $0$      & $1$      & $\geq 2$ & $2$    & (d)            \\
  $0$      & $\geq 2$ & $\bot$   & $\bot$ & (a)               \\
  $1$      & $0$      & $\bot$   & $0$    & $|X| = 1$       \\
  $1$      & $1$      & $1$      & $0$    & (c)             \\
  $1$      & $1$      & $\geq 2$ & $2$    & (d)            \\
  $1$      & $\geq 2$ & $\bot$   & $\bot$ & (a)               \\
  $2$      & $0$      & $\bot$   & $\bot$ & (b)              \\
  $2$      & $1$      & $1$      & $\bot$ & (b)              \\
  $2$      & $1$      & $\geq 2$ & $\bot$ & (b)              \\
  $2$      & $\geq 2$ & $\bot$   & $\bot$ & (a), (b)           \\
  $\geq 3$ & $0$      & $\bot$   & $2$    & (e)           \\
  $\geq 3$ & $1$      & $1$      & $2$    & (e)           \\
  $\geq 3$ & $1$      & $\geq 2$ & $2$    & (d), (e)      \\
  $\geq 3$ & $\geq 2$ & $\bot$   & $\bot$ & (a)               \\
  \hline
\end{tabular}
\end{center}
The subshift $X$ is a disjoint union of the subsystem consisting of cycles of cycle length $1$ (fixed points), which we call the \emph{unary part}, and the rest of the points, which we call the \emph{cycle part}. The reasons in the table are the following:
\begin{itemize}
\item[(a)] there are cycles of multiple lengths, which we can cycle independently
\item[(b)] there are exactly two unary points, so their transposition is in the center
\item[(c)] there is only one nontrivial cycle, so automorphisms are shifts
\item[(d)] the cycle part of $X$ requires $2$ automorphisms, and two suffice (see below)
\item[(e)] the unary part of $X$ requires $2$ automorphisms, and two suffice
\end{itemize}

Note that in the case $c_1 \geq 3, n = 1, c \geq 2$, naively dealing separately with the unary part and the cycle part of $X$, we need four generators -- two generators for each part. However, automorphisms do not mix these parts, so we need only two generators in total to reduce the centralizer to shifts, by using the same automorphisms to generate both parts. All of the observations used above are all either trivial or basic permutation theory, except (d) which may require a short explanation. This explanation is provided below.

By restricting to the cycle part of $X$, we may suppose that all cycles of $X$ have the same length $\ell \geq 2$ and there are $c \geq 2$ such cycles. We show that $k(X) = 2$. Let $X = \Z_\ell \times \Z_c$ where dynamics $\sigma : X \to X$ is given by incrementation in all the components $\Z_\ell \times \{i\}$.

Let $f : X \to X$ be any bijection commuting with $\sigma$. Then $f$ satisfies
\[ \forall n, i: \exists n',i': \forall j: f(n+j,i) = (n'+j,i'). \]
The lower bound $k(X) \geq 2$ follows since $\Aut(X)$ is non-abelian. If $c = 2$, then $k(X) = 2$, by taking $F = \{a,b\}$ where $a(n,0) = (n+1,0)$ and $a(n,1) = (n,1)$, and $b(n,i) = (n,i+1)$. Namely, consider any bijection $f : X \to X$ commuting with $a$, $b$ and $\sigma$. If $f(n,i) = (n',i')$ where $i = 0 \neq 1 = i'$, then $(a \circ f)(n,i) = a(n',i') = (n',i')$ and $(f \circ a)(n,i) = f(n+1,i) = (n'+1,i')$, so $f$ does not commute with $a$. It follows that $f$ preserves orbits, that is, $f(n,i) \in \Z_\ell \times \{i\}$ for all $n,i$. Commutation with $b$ implies that $f$ is a power of $\sigma$.

We also have $k(X) = 2$ for larger $c$: define $F = \{a, b\}$ where $a$ and $b$ are defined as follows: $a(n, 0) = (n, 1)$, $a(n,1) = (n,0)$ and $a(n, i) = (n, i)$ for $i \geq 2$, and $b(n,i) = (n,i+1)$ for all $n, i$. Consider any bijection $f : X \to X$ commuting with $a$, $b$ and $\sigma$. Observe that every permutation of the second component can be implemented with $a$ and $b$. This implies that for all $n, i$, we have $f(n,i) = (n',i)$ for some $n'$, since the symmetric group on $c$ elements has trivial center. By commutation with $b$, again $f$ is a power of $\sigma$.
\end{proof}

\subsection{Other examples of $\infty$-transitivity}

We gave two examples of $\infty$-transitivity in infinite groups in the introduction.

\begin{example}
Thompson's V is defined by its action on the half-open interval $[0,1)$. This action is well-defined on the countable set of dyadic rationals, and on this set the action is $\infty$-transitive.
\end{example}

\begin{proof}
We assume familiarity of the action of V on binary trees \cite{CaFlPa96}. Let $x_1, x_2, \ldots, x_k$ be distinct dyadic rationals, and take $n$ such that $x_i = k_i/2^n$ for all $i$ and some $k_i \in [0,2^n)$. Then a permutation $\pi$ of the $x_i$ can be implemented by permuting the $n$th level of the full binary tree. Every such permutation is in $V$.
\end{proof}

\begin{example}
The topological full group of an infinite minimal subshift $X$ is defined by its action on $X$. This action is well-defined on the (countable) shift-orbit of every point. There is a finitely generated subgroup of this group, namely its commutator subgroup, that is $\infty$-transitive on every shift-orbit.
\end{example}

\begin{proof}
The commutator subgroup is finitely generated \cite{Ma06}. We sketch a proof that the commutator subgroup is $\infty$-transitive on orbits directly, only using the assumption that $X$ is aperiodic. Fix a point $x \in X$. Let $k$ be arbitrary and let $n \in \N$. We show that we can permute the points $x^i = \sigma^{-i}(x)$ where $i \in [1, n]$ arbitrarily by an element of the commutator subgroup, proving $\infty$-transitivity. Let $\pi$ be any permutation of $[1, n]$.

First, we note that there exists a clopen set $C$ such that $x \in C$, and $y \in C$ implies $\sigma^{-\ell}(y) \notin C$ for all $\ell \in [1, n + 2]$. Namely, consider the word $w = x_{[-t', t']}$ for large $t'$. If $y_{[-t',t']} = w$ and $y_{[\ell-t',\ell+t']}$ for some $\ell$, then $x$ has a central pattern with period $\ell$ of length $t'-\ell$. Since $X$ contains no $\ell$-periodic points for any $\ell$, for any individual $\ell$ there must be an upper bound $t_\ell$ on $t'$. Pick $t$ to be the maximum of the $t_\ell$ for $\ell \in [1, n + 2]$. Then $C = \{y \in X \;|\; y_{[-t,t]} = x_{[-t,t]}\}$ is the desired clopen set for some large enough $t$.\footnote{Alternatively, the existence of $C$ can be seen by using the Marker Lemma \cite{LiMa95} and such clopen sets are also the building blocks of Bratteli-Vershik representations of minimal subshifts.}

Extend $\pi$ to an even permutation of $[1, n + 2]$ by acting either as a transposition or as identity on the two new elements. Given a point $y$, let $j(y) \in \N$ be minimal such that $\sigma^{j(y)}(y) \in C$ if such $j(y)$ exists (in the minimal case it always does) and otherwise $j(y) = \infty$. If $j(y) \in [1, n + 2]$, define
\[ f_\pi(y) = \sigma^{j(y) - \pi(j(y))}(y), \]
and otherwise define $f_\pi(y) = y$, so that $f_\pi$ permutes the points $\sigma^{-i}(y)$ for $i \in [1,n+2]$ according to $\pi$ whenever $y \in C$. Since the commutator subgroups of symmetric groups are the corresponding alternating groups, $f_\pi$ is easily seen to be in the commutator subgroup of the topological full group of $X$.
\end{proof}


Note that these examples are not about `orbit'-transitivity but normal transitivity. Orbit-transitivity is very specific to automorphism groups of dynamical systems (or more generally automorphisms of group-sets), and other types of groups would have their own `obvious restrictions'. For example, for linear groups we would require transitivity on linearly independent sets only, and for Thompson's F and Thompson's T we might only require transitivity on ordered sets of dyadic rationals rather than all tuples.

\subsection{Commutator subgroup action on pointed mixing SFT}

We sketch the proof that the commutator subgroup of the automorphism group of a pointed mixing SFT acts $\infty$-orbit transitively on the set of nonzero finite points. Here it is useful to have a characterization of $\infty$-orbit-transitivity in terms of permutations.

\begin{lemma}
\label{lem:InftyOrbitCharacterization}
Let $G$ and $H$ be groups whose actions commute on a set $X$, and let $X_0 \subset X$ be a subset of $X_0$ closed under the action of $H$ and containing infinitely many $H$-orbits. Then $G$ acts $\infty$-orbit-transitively on $X_0$ if and only if for every finite set $(x_1,\ldots,x_k)$ where the points $x_i$ are from different $H$-orbits and every even permutation $\pi : \{1,\ldots,k\} \to \{1,\ldots,k\}$ there exists $g \in G$ such that $g(x_i) = x_{\pi(i)}$ for all $i$.
\end{lemma}

\begin{proof}
Let $(x_1,\ldots,x_k)$ and $(y_1,\ldots,y_k)$ be $k$-tuples of elements of $X_0$ such that the elements $x_i$ are pairwise in different $H$-orbits, and the same holds for the $y_i$ as well, and we want to map $x_i$ to $y_i$ for all $i$. Since there are infinitely many disjoint $H$-orbits, we may assume $k$ is even.

Let $(z_1,\ldots,z_k)$ be points with $H$-orbits distinct from those of the points $x_i$ and $y_j$. The transformation $(x_1,\ldots,x_k) \mapsto (z_1,\ldots,z_k)$ is a restriction of the even permutation of $(x_1,\ldots,x_k,z_1,\ldots,z_k)$ that transposes $x_i$ and $z_i$ for all $i$, and thus by the assumption there exists $g \in G$ such that $g(x_i) = z_i$ for all $i$. We can deal with the transformation $(z_1,\ldots,z_k) \mapsto (y_1,\ldots,y_k)$ similarly, showing that $G$ is $\infty$-orbit-transitive on $X_0$.
\end{proof}

Note that we apply the assumption twice, and indeed we typically need to: the transformation $x \mapsto h x$ cannot be implemented by a permutation of a finite set of points from distinct $H$-orbits by an $H$-commuting action when $h$ is not a torsion element in $H$, as $g x = h x \implies g(h^k(x)) = h^k(g(x)) = h^{k+1}(x)$ so necessarily the orbit of $x$ is infinite if $g$ maps it to $hx$.

We also need the following sufficient condition for safety.

\begin{lemma}
Let $\Sigma \ni 0$, let $V = \Sigma^n \setminus \{0^n\}$ and let $U \subset 0^n \Sigma^n 0^n \setminus \{0^{3n}\}$ satisfy
\[ 0^{n+j} u 0^{n+j'}, 0^{n+\ell} u 0^{n+\ell'} \in U \implies j = \ell. \]
Then $U$ is a $V$-safe set of words.
\end{lemma}

\begin{proof}
The proof is similar to that of Lemma~\ref{lem:SufficientSafety}. The first condition of safety is clearly satisfied, since $0^{3n}$ is not in $U$ and $V$ contains all nonzero words. To see that the second is satisfied, note that points $x$ such that $\chi_U(x) = \{0\}$ and $\chi_V(x) \subset [0,2n]$, which are considered in the definition of safety, in fact contain only $0$s outside the interval $[0,3n-1]$. If we change the central $U$-word, and obtain a second occurrence of a word of $U$, as in Lemma~\ref{lem:SufficientSafety} the distance between the two occurrences would be at most $n$. This violates the assumption of the lemma. Similarly, we see that the points $...000. u 000...$, $u \in U$, prove the third condition of safety.
\end{proof}

\begin{theorem}
\label{thm:CommutatorAction}
Let $X \subset \Sigma^\Z$ be a pointed mixing SFT. Then the commutator subgroup of $\Aut(X)$ acts $\infty$-orbit-transitively on the nonzero finite points.
\end{theorem}

\begin{proof}
Let $(x_1,\ldots,x_k)$ be nonzero finite points with distinct $\sigma$-orbits, and suppose $k \geq 5$. It is enough to show that for any even permutation $\beta : [1,k] \to [1,k]$ there is an automorphism in the commutator subgroup that performs the corresponding permutation. The result then follows from Lemma~\ref{lem:InftyOrbitCharacterization}.

Let $w_i = (x_i)_{[-3m,3m-1]}$ for all $i$, where $m$ is large enough so that the words $w_i$ contain the supports of the point $x_i$ and such that for all $i$, $w_i = 0^{2m} u_i 0^{2m}$ for some $u_i$, and $m$ is larger than the window size of $X$.

Now, picking $n = 2m$, $U = \{w_1, w_2, \ldots, w_k\}$, and $V = \Sigma^n \setminus \{0^n\}$, by the previous lemma $U$ is $V$-safe (the assumption of the lemma being satisfied because the points $x_i$ are from distinct orbits) and there is an automorphism of $X$ implementing any $(n-1)$-safe permutation (thus in fact any permutation) of $U$ by Lemma~\ref{lem:TechnicalSafety}.

Since $\beta$ is an even permutation and $k \geq 5$, $\beta$ can be written as a composition of commutators. Thus, the same is true in the automorphism group of $X$.
\end{proof}

A corollary of this result is that if $k(X) = \infty$ for some mixing SFT $X$ containing $0^\Z$, then the commutator subgroup of $\Aut(X)$ cannot be finitely generated.

From the theorem above, one `almost' obtains Ryan's theorem for the commutator subgroup (see the next section). Ryan's theorem is proved for the subgroup of commutator subgroup generated by commutators of involutions in \cite[Proposition~4.11]{BoCh17}.

\subsection{Making generators commutators}

In this section we explain the minor changes needed to show that a group with the transitivity properties we want can also be obtained in the commutator, in view of Theorem~\ref{thm:CommutatorAction} and Question~\ref{q:Commutator}. We also show that all generators can be written as a composition of involutions.

\begin{theorem}
\label{thm:InvoCommu}
For $\Sigma = \{0,1,2,3\}$ there is a finite set $F$ such that $\langle F \rangle$ acts $\infty$-orbit-transitively on the set of nonzero finite points $x \in \Sigma^\Z$, and each element of $F$ is a commutator of two automorphisms, and a composition of involutive automorphisms.
\end{theorem}

\begin{proof}[Proof sketch]
The $\infty$-orbit-transitive group constructed in Theorem~\ref{thm:MainProof} is
\[ \mathcal{G} = \langle \mathcal{G}_{3}, P, g_{(13)}, g_{(23)} \rangle \leq \Aut(\Sigma^\Z). \]
Most of the generators are not of the required form; $g_{(13)}$ is an involution, but it is not in the commutator subgroup because it performs an odd permutation on $\{0^\Z,1^\Z,2^\Z,3^\Z\}$ (which means it has a nontrivial image in some abelian quotient of $\Aut(\Sigma^\Z)$). For example the homomorphism constructed in \cite{Ka96} gives a nontrivial image for $P$ in a torsion-free abelian group, implying that $P$ is not in the commutator, and is not generated by involutions.

We revisit every step of our construction to obtain generators of the required form. We need the combinatorial fact about the symmetric group $S_4$ that an involution consisting of two disjoint cycles is a commutator:
\[ (0 \; 1)(2 \; 3) = [(0 \; 3)(1 \; 2), (0 \; 3 \; 1 \; 2)] = (0 \; 3)(1 \; 2) \circ  (0 \; 3 \; 1 \; 2) \circ (0 \; 3)(1 \; 2) \circ (0 \; 3 \; 1 \; 2)^{-1}. \]

We start with the particle rule $P$. Let $Q$ be the rule the moves particles to the right and walls to the left, that is, $Q = Q' \circ Q''$ where $Q'$ is the transposition $(1 \; 2)$ (which transposes particles and walls) and $Q'' = P \circ Q' \circ P^{-1}$. This is not immediately seen to be in the commutator, but $Q^2 = Q' \circ Q'' \circ Q' \circ Q'' = [Q', Q'']$ is. We replace $P$ by $Q^2$ in our set of generators.

Assuming that other generators are not changed, $\mathcal{G}_{3}$ is generated by our generators. We conjugate the generators of $\mathcal{G}_3$ by $P$, $P^2$ and $P^3$ to obtain groups $\mathcal{G}_{3}^2$, $\mathcal{G}_{3}^3$ and $\mathcal{G}_{3}^4$ which act just like $\mathcal{G}_3$, but in `heads', there is a offset $1$, $2$ or $3$ between the particle and the wall.

We can now go through the proof of the main theorem: we first construct a good vector using $g_{(13)}$ and $g_{(23)}$ and $Q^2$, and once we have a good vector we apply $Q^{-2}$ until the rightmost particle of some configuration in our vector gets within distance at most $3$ to the left of a wall. With four applications of Lemma~\ref{lem:AlmostTransitively}, using the correct conjugate of $\mathcal{G}_3$, we can perform any operations simultaneously in all coordinates of the vector where this happens. In particular, we can freely organize positions and times where a head appears when $Q^2$ is applied repeatedly, and move all heads far enough to the right not to affect subsequent steps. This shows that we can replace $P$ by $Q^2$ in the generating set, and $Q^2$ is a commutator generated by involutions.

While $g_{(13)}$ and $g_{(23)}$ are not in the commutator subgroup, we note that the only important property of $g_{(23)}$ is that it allows turning prepregood vectors into pregood ones. For this we need that in a word containing only walls it introduces at least one particle, does not remove every wall, and does not eliminate every particle in a word containing only particles.

We can reproduce the behavior of $g_{(23)}$ as follows: Consider a prepregood vector $\vec x \in X_0^{(k)}$. First, apply the automorphism that performs the permutation $(0200 \; 0210) \circ (0220 \; 0230)$ in every coordinate where it applies. A short case analysis shows that this rule gives a well-defined involutive automorphism, which can be written as a commutator as in the second paragraph of the proof. When applied to a subword $w \in \{0,2\}^*$ of a configuration, it introduces a particle to the right of, or on top of, the rightmost wall, unless there are three walls at the right end of the word. We then repeatedly apply the particle rule. After this, all points that contain walls and do not contain at least three walls as their rightmost symbols contain particles, and still every coordinate of $\vec x$ is prepregood. Then use the automorphism implementing the permutation $(22000 \; 22010) \circ (22020 \; 22030)$. This is again well-defined and is a commutator, and it introduces particles in all coordinates of $\vec x$ that contain(ed) at least two walls as their rightmost symbols. After again applying the particle rule, we get that every coordinate of the resulting vector contains a particle, giving a pregood vector. To get good vectors, we replace $g_{(13)}$ by two automorphisms in a symmetric way.

As for the replacement of $\mathcal{G}_{3}$, we know that it is enough to take $f_{\pi,s}$ with $\pi : A^\ell \to A^\ell$ taken from a finite set of even permutations, by Lemma~\ref{lem:Gates} and a slight modification of the proof of Lemma~\ref{lem:TMTransitive}. The only problem is that the coordinate swap $(ab \; ba)$, which is needed to apply rules to arbitrary subsequences, does not give an even permutation of words (since our alphabet is of odd cardinality). However, one can use even coordinate permutations to obtain a finite generating set for $\langle G_0'', \sigma \rangle$, and $G_0''$ already acts $\infty$-transitively by Lemma~\ref{lem:TMTransitive}. Even permutations are generated by pairs of disjoint two-cycles when $\ell$ is large enough, since large enough alternating groups are simple and cycle type is conjugacy invariant, so each $f_{\pi,s}$ can be written as a composition of involutions that are in the commutator as in the second paragraph of the proof.

Finally we replace $f_{\sigma,s}$ by a map that is a composition of involutions that are commutators. For this, we retrace the steps of the proof of Lemma~\ref{lem:SimulatedShift}, but when defining $\pi$ we pick a large $m$ and use four sets: $U_1 = \{A^{m+1} s A^{m-1}\}$, $U_2 = \{A^{m} s A^k s A^{m-1-k} \;|\; k \in \{0,\ldots,n-2\}\}$, $U_3 = \{A^{m} s A^{n+k} s A^{m-n-1-k} \;|\; k \in \{0,\ldots,n-2\}\}$, and $U_4 = \{A^{m} s A^{2n+k} s A^{m-2n-1-k} \;|\; k \in \{0,\ldots,n-2\}\}$.
All of these sets have cardinality $(n-1)^{2m}$, and are disjoint. Now we can have $\pi$ act on $U_1 \cup U_2$ as an involution, mapping $\pi(U_i) = U_{3-i}$, and similarly defining $\pi$ on $U_3 \cup U_4$ in an involutive way, to get an involution $\pi$ on $U_1 \cup U_2 \cup U_3 \cup U_4$.

The lengths $2m+1$ can be chosen to be divisible by three, and for large enough $m$, Lemma~\ref{lem:SufficientSafety} applies, and we can perform $\pi$ by a local rule by Lemma~\ref{lem:TechnicalSafety}. On $U$, $\pi$ can be written as a composition of pairs of cycles, and thus as a composition of involutions that are also commutators, as in the second paragraph of the proof. Thus the automorphism $f_\pi$, obtained from Lemma~\ref{lem:TechnicalSafety}, can be written in terms of automorphisms that are involutions and commutators because the map $\gamma$ obtained from the lemma is a group homomorphism. We deal with $\tau$ similarly, obtaining a map $f$ which is a composition of involutions that are commutators, and acts as $f_{\sigma,s}$ on $Y_{1,s}$.
\end{proof}

We also obtain a version of Theorem~\ref{thm:Ryan} in this setting. 



\begin{theorem}
\label{thm:RyanVersion}
For $\Sigma = \{0,1,2,3\}$ there is a finite set $F \subset \Aut(\Sigma^\Z)$ such that each element of $F$ is a commutator of two automorphisms, and a composition of involutive automorphisms, and if $f$ commutes with every $g \in F$, then $f \in \langle \sigma \rangle$.
\end{theorem}

\begin{proof}
If $f$ commutes with the cellwise symbol-permutation $(1 \; 2 \; 3)$ (which is a commutator of two automorphisms and a composition of involutive automorphisms), then it maps $0^\Z$ to $0^\Z$. We conclude as in the proof of Theorem~\ref{thm:Ryan} using Theorem~\ref{thm:InvoCommu}.
\end{proof}

\section*{Acknowledgements}

I thank Thomas Worsch for discussions on transitivity and reversibility of the action of the monoid of cellular automata, which lead to the study of the corresponding questions for the automorphism group. I thank the referee for their comments, in particular for spotting an error in the proof of the predecessor of Lemma~\ref{lem:AlmostTransitively} and several smaller mistakes, which led to the rewriting of many parts of the paper, and for suggesting that I prove Theorem~\ref{thm:CommutatorAction} (which was part of Question~\ref{q:Commutator} originally). I thank Jarkko Kari for Lemma~\ref{lem:SimulatedShift}, which was originally a question -- this led to a serious rewrite, which in turn led to the addition of Theorem~\ref{thm:InvoCommu}. I thank Johan Kopra for suggesting the symbol permutation $(1 \; 2 \; 3)$ in the last theorem. I thank the support of COST Action IC1405.

\bibliographystyle{plain}
\bibliography{bib}{}

\end{document}